\documentclass{amsart}

\usepackage{amssymb}
\usepackage{amsthm}
\usepackage{amsmath}
\usepackage{mathrsfs}			
\usepackage[usenames,dvipsnames]{color}
\usepackage[unicode]{hyperref}	
\hypersetup{urlcolor=blue, citecolor=blue, linkcolor=blue, colorlinks=true}		
\usepackage{geometry}
\usepackage{subfigure}
\usepackage{tikz-cd}
\usepackage[bf]{caption}
\input xy
\xyoption{all}

\renewcommand{\phi}{\varphi}

\newcommand{\QQ}{\mathbf{Q}}
\newcommand{\ZZ}{\mathbf{Z}}
\newcommand{\RR}{\mathbf{R}}
\newcommand{\CC}{\mathbf{C}}

\newcommand{\PP}{\mathbf{P}}

\newcommand{\ds}{\displaystyle}

\newcommand{\funcdef}[5]{\begin{array}{rcl} #1 : #2 & \rightarrow & #3 \\ #4 & \mapsto & #5 \end{array} }

\newcommand{\ol}[1]{\overline{#1}}
\newcommand{\mf}[1]{\mathfrak{#1}}
\newcommand{\mc}[1]{\mathcal{#1}}
\newcommand{\wt}[1]{\widetilde{#1}}
\newcommand{\gl}{\mathrm{GL}}
\newcommand{\GO}{\mathrm{GO}}
\newcommand{\SO}{\mathrm{SO}}

\renewcommand{\mod}[1]{\text{\! }(\operatorname{mod}\text{\! }#1)}
\renewcommand{\varepsilon}{\epsilon}

\DeclareMathOperator{\sgn}{sgn}

\DeclareMathOperator{\tr}{tr}
\DeclareMathOperator{\Tr}{Tr}

\DeclareMathOperator{\SL}{SL}
\DeclareMathOperator{\Vol}{Vol}
\DeclareMathOperator{\sh}{sh}
\DeclareMathOperator{\ord}{ord}

\newcommand{\VQ}{V^{(Q)}}
\newcommand{\vQ}{v^{(Q)}}
\newcommand{\xQ}{x^{(Q)}}
\newcommand{\yQ}{y^{(Q)}}
\newcommand{\zQ}{z^{(Q)}}
\newcommand{\FQ}{\mc{F}^{(Q)}}
\newcommand{\Or}{{\mathrm{Or}}}

\renewcommand{\Re}{\mathrm{Re}}
\renewcommand{\Im}{\mathrm{Im}}

\newcommand{\vect}[1]{\begin{pmatrix}#1\end{pmatrix}}

\newtheorem{theorem}{Theorem}[section]
\newtheorem{theoremA}{Theorem}
\newtheorem{proposition}[theorem]{Proposition}
\newtheorem{lemma}[theorem]{Lemma}
\newtheorem{corollary}[theorem]{Corollary}

\numberwithin{equation}{section}

\newcommand{\TZP}{T_K^{\perp\prime}}
\newcommand{\TO}{T_\mc{O}^{\perp}}
\newcommand{\MO}{M_\mc{O}^{\perp}}
\newcommand{\A}{\mc{A}}
\newcommand{\im}{\mathrm{im}}
\newcommand{\ratio}{\mathrm{ratio}}

\hypersetup{pdfauthor={Robert Harron},pdftitle={Equidistribution of shapes of complex cubic fields of fixed quadratic resolvent},pdfkeywords={Algebraic number theory, cubic fields, lattices, arithmetic statistics, equidistribution, geodesics}}

\newcommand{\Rcom}[1]{}

\begin{document}

\title[Equidistribution of shapes of complex cubic fields]{E\MakeLowercase{quidistribution of shapes of complex cubic fields of fixed quadratic resolvent}}

\author{R\MakeLowercase{obert} H\MakeLowercase{arron}}
\address{
Department of Mathematics\\
Keller Hall\\
University of Hawai`i at M\={a}noa\\
Honolulu, HI 96822\\
USA
}
\email{rharron@math.hawaii.edu}
\thanks{The author is partially supported by a Simons Collaboration Grant.}
\date{\today}

\begin{abstract}
We show that the shape of a complex cubic field lies on the geodesic of the modular surface defined by the field's trace-zero form. We also prove a general such statement for all orders in \'{e}tale $\QQ$-algebras. Applying a method of Manjul Bhargava and Piper~H to results of Bhargava and Ariel Shnidman, we prove that the shapes lying on a fixed geodesic become equidistributed with respect to the hyperbolic measure as the discriminant of the complex cubic field goes to infinity. We also show that the shape of a complex cubic field is a complete invariant (within the family of all cubic fields).
\end{abstract}

\maketitle

 
\tableofcontents
 
\section{Introduction}

David Terr, in his PhD thesis \cite{terr}, introduced the notion of the \textit{shape} of a number field $K$. This is a certain lattice (of rank $[K:\QQ]-1$) attached to $K$ considered up to rotation, reflection, and scaling. In \cite{terr}, he proves that the shapes of cubic fields are equidistributed in the space of shapes of rank $2$ lattices (as the discriminant of the cubic field goes to infinity). Manjul Bhargava and Piper~H (\cite{Manjul-Piper,PiperThesis}) generalize this result to show the equidistribution of shapes of $S_n$-number fields of degree $n$ for $n=4$ and $5$. These authors in fact conjecture that such equidistribution holds for all $n$ reflecting the idea that a degree $n$ number field with Galois group $S_n$ is ``random''. The primary goal of this article is to investigate the distribution of more specific (read ``less random'') families of number fields. In \cite{terr}, Terr proves that $C_3$-cubic fields all have the same shape (hexagonal!), thus showing that restricting the number fields can impose strong constraints on the shapes. In \cite{Manjul-Ari}, the shapes of real cubic fields with fixed quadratic resolvent field are shown to lie in a finite set and be (essentially) equidistributed in it. We consider the shapes of complex cubic fields of fixed quadratic resolvent and show that they are equidistributed on certain geodesics on the modular surface (see Fig.~\ref{fig:equid5} for an example). Combined with \cite{PureCubicShapes}, our main result suggests an intriguing explanation of a result of Cohen--Morra \cite[Theorem~1.1]{Cohen-Morra} and \cite[Theorem~6]{Manjul-Ari} that the number of cubic fields with quadratic resolvent $\QQ(\sqrt{-3})$ has bigger growth rate than that of fields with other quadratic resolvents by a log factor. We prove some other interesting results along the way, including that the shape of a complex cubic determines that field within the family of all cubic fields, and that the shapes of degree $n$ number fields with given trace-zero form lie on the majorant space of the trace-zero form.

\begin{figure}[h]
	\includegraphics[scale=0.35]{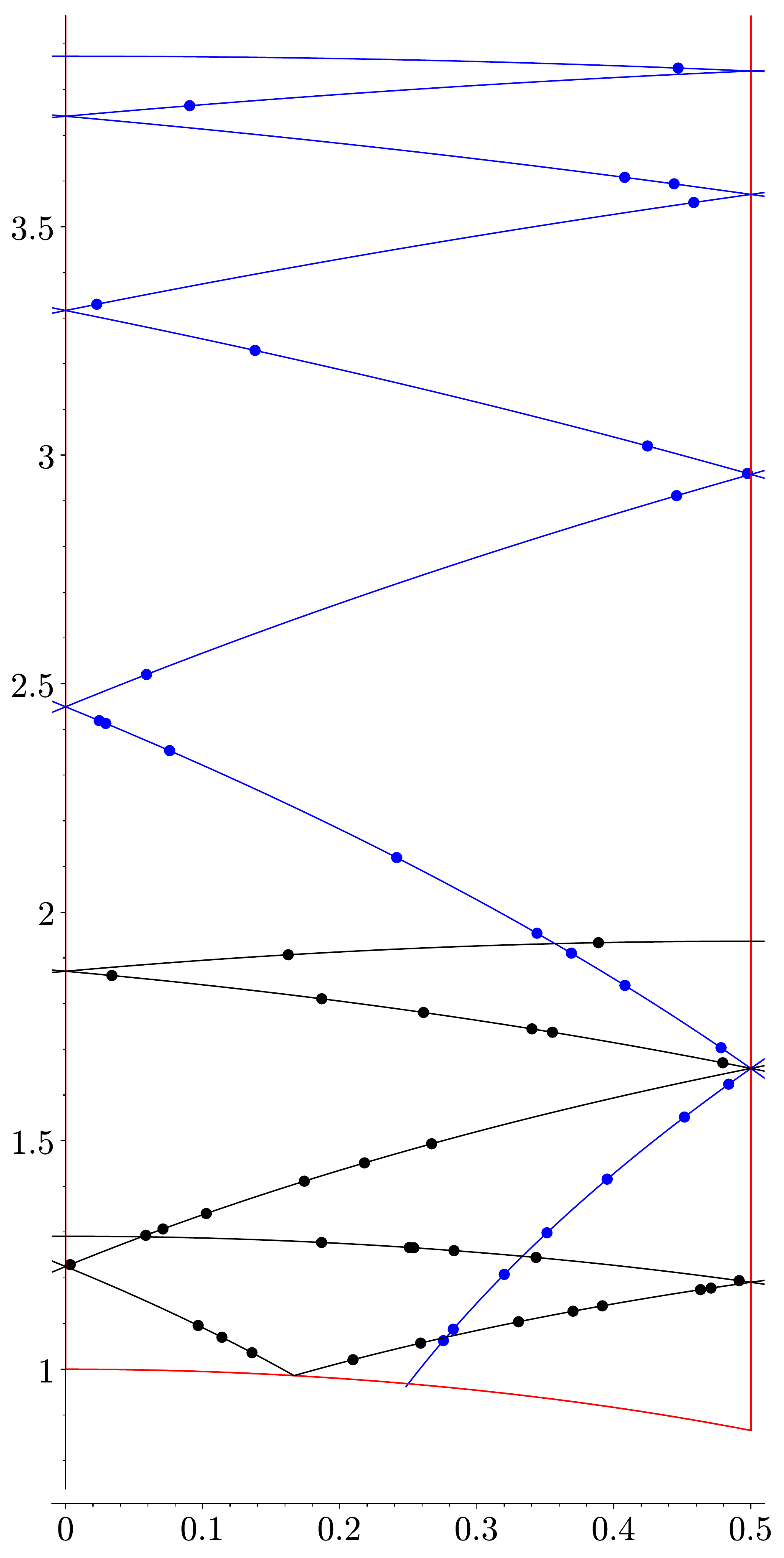}
	\caption{Shapes of all complex cubic fields $K$ with quadratic resolvent $\QQ(\sqrt{-20})$ and $|\Delta(K)|\leq3,375,000$. The blue geodesic (i.e.\ the one that goes higher) corresponds to the class of $-x^2+6xy+6y^2$, whereas the black geodesic corresponds to that of $-2x^2+6xy+3y^2$. These represent the only two $\gl_2(\ZZ)$-classes of indefinite integral binary quadratic forms of discriminant ${60=(-3)\cdot(-20)}$.}
	\label{fig:equid5}
\end{figure}

We now describe in more detail our results and how they fit in to our current understanding.

\subsection{The shape as an invariant}
A fundamental use of attaching invariants to number fields (or to any objects in mathematics) is to help distinguish between them. The degree and the discriminant are the first one typically encounters and together they completely distinguish a quadratic number field from other number fields. Sadly, in every degree greater than 2, there are non-isomorphic fields with the same discriminant. A natural refinement of the discriminant of a degree $n$ number field $K$, with ring of integers $\mc{O}_K$, is the isometry class of the \textit{trace pairing}
\[\funcdef{T_K}{\mc{O}_K\times\mc{O}_K}{\ZZ}{(a,b)}{\Tr_{K/\QQ}(ab)},
\]
where $\Tr_{K/\QQ}:K\rightarrow\QQ$ denotes the usual field trace. Indeed, the discriminant is the determinant of this pairing (i.e.\ the determinant of a Gram matrix representing this bilinear form). Despite some goods news in the case of totally real fields (see e.g. \cite{Casimir}), the extra information of the trace form is no help for complex cubic fields: two complex cubic fields have isometric trace forms if and only if they have the same discriminant (\cite[Theorem 3.3]{M-S})

The \textit{shape} of a number field stems from a similar refinement of the discriminant using Minkowski's geometry of numbers. One obtains a pairing
\[
 M_K:\mc{O}_K\times\mc{O}_K\rightarrow\ZZ
\]
by first embedding $K$ into $\CC^n$ and taking the standard Hermitian inner product on $\CC^n$. Unlike $T_K$ whose signature depends on the number of real and complex places of $K$ (\cite{Taussky}), $M_K$ is always positive-definite. In this article, we prove results that seem to indicate this feature allows it to retain more information about the number field than the trace form does. More specifically, we do not consider $M_K$, but rather its projection $M_K^\perp$ to the trace zero space, up to homotheties. In other words, we take the lattice corresponding to $M_K$, project it onto the space of elements of $K$ of trace zero, and consider its equivalence class under scaling, rotations, and reflections. This is what is called the \textit{shape} of $K$, denoted $\sh(K)$. In \S\ref{sec:complete_invariant}, we prove that the shape is a complete invariant in the family of complex cubic fields.

\begin{theoremA}\label{thm:complete_invariant}
	If $K$ is a complex cubic field and $L$ is any other cubic field, then $\sh(K)\neq\sh(L)$. In particular, the shape is a complete invariant in the family of complex cubic fields.
\end{theoremA}

Since the shape of $K$ is a rank 2 lattice up to homothety, it corresponds to a point in the upper-half plane. We in fact show that the field $K$ is obtained by adjoining to $\QQ$ the $x$-coordinate of its shape (or, when $K$ is a pure cubic\footnote{Recall that a \textit{pure} cubic field is one of the form $\QQ(\sqrt[3]{m})$ for some non-cube $m\in\QQ$.}, its $y$-coordinate). This is then reminiscent of how a quadratic field is obtained by adjoining to $\QQ$ the square root of its discriminant.

\subsection{The distribution of complex cubic shapes}

Once we have an invariant, we can ask how ``random'' it is. A general philosophy is that it should be as random as it can be! And this is what we show.

What constraints are there on the shape of a complex cubic field $K$? The following result gives an elegant answer involving the association of a geodesic on $\gl_2(\ZZ)\backslash\mf{H}$ to an indefinite binary quadratic form over $\ZZ$ (see \S\ref{sec:geodesics} for some background about these geodesics and the beginning of \S\ref{sec:equidistribution} for a proof of this theorem).

\begin{theoremA}\label{thm:3TKonMK}
	The shape of the complex cubic field $K$ lies on the geodesic associated to the indefinite binary quadratic form $T_K^\perp$. 
\end{theoremA}
Here, $T_K^\perp$ is the projection of $T_K$ to the trace zero space.

In fact, in \S\ref{sec:higherdegree}, we prove a generalization of this to all number fields. The shape of $K$ is a point in the space of rank $n-1$ lattices and a construction of Siegel's associates to a quadratic form $Q$ a subspace $\mf{H}_{Q}$ of this space. We prove the following theorem.
\begin{theoremA}\label{thm:higherdegree}
	If $K$ is a number field (of degree $\geq3$), then its shape lies on $\mf{H}_{T_K^\perp}$.
\end{theoremA}
When $K$ is not totally real, $\mf{H}_{T_K^\perp}$ is positive-dimensional, whereas $\sh(K)$ is a point. It is in this sense that we say that the positive-definiteness of $M_K^\perp$ seems to retain more information than $T_K^\perp$.

Considering $T_K^\perp$ as a quadratic form over $\ZZ$, we can study the associated primitive quadratic form $\TZP$ obtained by dividing by the gcd of the coefficients. We can characterize which (equivalence classes of) quadratic forms arise as $[\TZP]$ for some cubic field $K$ (here we denote by $[Q]$ the equivalence class of the binary quadratic form $Q$ under the action of $\gl_2(\ZZ)$). The work of Bhargava and Shnidman \cite[\S5]{Manjul-Ari} tells us most of the story, though the following result, which we prove in \S\ref{sec:TameVsWild}, adds the finishing touch.
\begin{theoremA}\label{thm:TameVsWild}
	If $K$ is a cubic field such that the discriminant of its quadratic resolvent field is divisible by $3$, then
	\[
		\ord_3\Delta\!\left(\TZP\right)=\begin{cases}
				0 & \text{if }3\text{ is wild in }$K$,\\
				2 & \text{if }3\text{ is tame in }$K$.
			\end{cases}
	\]
\end{theoremA}
We refer to this phenomenon as the Tame--Wild Dichotomy. With this and \cite[\S5 and Theorem~4]{Manjul-Ari} in hand, we can say the following.
\begin{corollary}
	Let $K$ be a cubic field whose quadratic resolvent field has discriminant $d$. Let
	\begin{equation}
		D=\begin{cases}
			-3d& \text{if $3$ is tame in }$K$\\
			-d/3& \text{if $3$ is wild.}
		\end{cases}
	\end{equation}
	Then, $\TZP$ has discriminant $D$. Furthermore, as $K$ varies over fields with fixed $d$ and fixed choice of whether $3$ is tame or wild in $K$, the $\TZP$ of such $K$ are equidistributed amongst the finitely many equivalence classes\footnote{Because of ambiguous forms, to get equidistribution we must either count \textit{oriented} cubic fields and use $\SL_2(\ZZ)$-equivalence of binary quadratic forms, or weight the forms by the ``size'' of their automorphism group (ambiguous forms having automorphisms groups that are twice as big).} of binary quadratic forms of discriminant $D$ as the discriminant of $K$ goes to infinity.
\end{corollary}

	We note that, in \cite{Manjul-Ari}, the authors refer to $\TZP$ as the ``shape of $K$'' even when $K$ is not real. This is in contrast to the original definition of the shape in \cite{terr}. In the real case, they are indeed discussing the shape of $K$ and so the above corollary tells us about the distribution of shapes of real cubic fields with fixed quadratic resolvent.

After this result, what remains is to understand the distribution of shapes of complex cubic fields with fixed $[\TZP]$. We prove that they are equidistributed on the geodesic corresponding to $[\TZP]$ with respect to its natural hyperbolic measure as the discriminants of the fields go to infinity. For pure cubic fields, i.e.\ those fields whose $\TZP$ is (equivalent to) $xy$ or $3xy-y$, this is proved in the author's previous article \cite[Theorem~C]{PureCubicShapes}. In that case, the corresponding geodesics have \textit{infinite} length and the equidistribution is in a ``regularized'' sense as described in \textit{ibid.} For non-pure complex cubic fields, the geodesics in question all have finite length. We will prove the following result in \S\ref{sec:equidistribution}.

\begin{theoremA}\label{thm:Equidistribution}
	Let $[Q]$ be a(n equivalence class of a) quadratic form that arises as the $[\TZP]$ of some non-pure complex cubic field $K$. Let $[\gamma_Q]$ denote the associated geodesic in $\gl_2(\ZZ)\backslash\mf{H}$ endowed with the measure $\mu_Q$ it inherits from the hyperbolic metric on $\mf{H}$. Let $W$ be a $\mu_Q$-continuity set\footnote{Recall that a \textit{continuity set} is a measurable subset whose boundary has measure $0$.} of $[\gamma_Q]$ and let
	\[
		N(Q;X,W)=\{K\text{ complex cubic field}:[\TZP]=[Q], \sh(K)\in W, |\Delta(K)|<X\}.
	\]
	Then,
	\begin{equation}
		\lim_{X\rightarrow\infty}\frac{N(Q;X,W)}{N(Q;X,[\gamma_Q])}=\frac{\mu_Q(W)}{\mu_Q([\gamma_Q])},
	\end{equation}
	i.e. the shapes are equidistributed as the discriminant of $K$ goes to infinity. More specifically, there is a constant $C_Q>0$ such that
	\begin{equation}\label{eqn:MainGrowthFormula}
		N_Q(X;W)=C_Q\mu_Q(W)\sqrt{X}+o(\sqrt{X}).
	\end{equation}
\end{theoremA}
In \cite[Theorem~C]{PureCubicShapes}, what is proven is that when $K$ is pure cubic, equation~\eqref{eqn:MainGrowthFormula} holds for all \textit{compact} $\mu_Q$-continuity sets $W$. We prove analogues of Theorem~\ref{thm:Equidistribution} along the way for orders, as well as oriented fields and orders in Theorem~\ref{thm:equidorientedorders}, Corollary~\ref{cor:equidorientedfields}, and Theorem~\ref{thm:equidorders}, respectively.

As an additional bit of information on the location of shapes of complex cubic fields, we have the following result that we prove in \S\ref{sec:avoidboundary} and that we use in the proof of Theorem~\ref{thm:Equidistribution}.
\begin{theoremA}\label{thm:avoidboundaries}
	The only complex cubic fields whose shape lies on the boundary of the space of two-dimensional shapes (or, equivalently, whose lattice has extra automorphisms) are the wild pure cubic fields.
\end{theoremA}

\subsection{Shapes and log terms in asymptotics of counting number fields}
Aside from the inherent interest in understanding the distribution of shapes of number fields, the equidistribution result above together with that of \cite{PureCubicShapes} suggest an interesting explanation for the occurrence of log terms in the asymptotics of counting number fields.

A theorem of Henri Cohen and Anna Morra \cite[Theorem~1.1]{Cohen-Morra}, and independently Bhargava and Shnidman \cite[Theorem~6]{Manjul-Ari}, says that the number of cubic number fields $K$ of fixed quadratic resolvent field $K_2$ (and discriminant bounded by $X$) grows like
\[
	\begin{cases}
		\sqrt{X}&\text{if }K_2\neq\QQ(\sqrt{-3})\\
		\sqrt{X}\log(X)&\text{if }K_2=\QQ(\sqrt{-3}).
	\end{cases}
\]
What we show in this article and \cite{PureCubicShapes} is that for a fixed indefinite quadratic form $Q$ (that occurs as $\TZP$ for some complex cubic field $K$), there is a constant $C_Q>0$ such that for any \textit{compact} continuity set $W$ in $[\gamma_Q]$,
\[
	N_Q(X,W)=C_Q\mu_Q(W)\sqrt{X}+o(\sqrt{X}),
\]
\textit{independent} of the quadratic resolvent field of $K$. The quadratic resolvent field is $\QQ(\sqrt{-3})$ if and only if the field is a pure cubic. In this case, the geodesics in question have infinite length. For instance, for a wild pure cubic field $K$, the shape is rectangular and is thus parametrized by what we call the \textit{ratio} $r_K$ of $K$ that measures the ratio of the length of the sides of the rectangle. The natural measure on the space of rectangular lattices is $\dfrac{dr_K}{r_K}$, so that the measure of the set of rectangles with ratio in the interval $[1,R]$ is proportional to $\log(R)$. We can then think of the extra $\log(X)$ term coming from the measure of the sets $W$ going to infinity.

The author has explored this phenomenon in joint work in other situations.

\begin{itemize}
	\item In \cite{QuarticShapes} with Piper H, we study the shapes of $V_4$-quartic fields. Baily showed (\cite{Baily}) that there is a constant $C_{V_4}>0$ such that the number of $V_4$-quartic fields with discriminant bounded by $X$ is
\[
	C_{V_4}\sqrt{X}\log^2(X)+o(\sqrt{X}\log^2(X)).\footnote{Secondary terms have been determined since Baily's work.}
\]
We show that the shapes of $V_4$-quartic fields $K$ are equidistributed (in a regularized sense) in a two-dimensional space of shapes of infinite measure. A bit more specifically, the shapes are given by rectangular prisms (so-called \textit{orthorhombic} lattices) whose side lengths are in ratios given by the discriminants of the three quadratic subfields of $K$. We may therefore parametrize the shapes by two ratios $r_{K,1}$ and $r_{K,2}$ so that the relevant space of lattices has measure $\dfrac{dr_{K,1}dr_{K,2}}{r_{K,1}r_{K,2}}$. The number of $K$ of discriminant bounded by $X$ whose shape lies in a compact continuity set $W$ is then shown to be proportional to
\[
	\mu(W)\sqrt{X}+o(\sqrt{X}).
\]
Similarly to the case of pure cubic fields, the set of orthorhombic lattices with sides ratios in the box of side $R\times R$ is proportional to $\log^2(R)$ so that we now see two log terms arising. This result is generalized to cases of triquadratic fields in the upcoming PhD thesis of Jamal Hassan.
	\item In joint work with Erik Holmes (\cite{SexticShapes}), we study shapes of sextic fields containing a quadratic subfield. This case is of particular interest as Kl\"{u}ners' counterexample \cite{Kluners} to Malle's conjecture is about fields of this form. Indeed, Kl\"{u}ners studies sextic fields $K$ whose Galois group is $C_3\wr C_2(\cong S_3\times C_3)$. These fields are exactly those non-Galois sextic fields containing some quadratic subfield $K_2$ over which they are Galois. Malle's original conjecture predicts that the number of such $K$ of discriminant bounded by $X$ grows like $\sqrt{X}$; however Kl\"{u}ners shows that even just the number of those with $K_2=\QQ(\omega)$ (where $\omega$ is a primitive cube root of unity) grows like $\sqrt{X}\log(X)$. In fact, Kl\"{u}ners shows that the number of $K$ with $K_2\neq\QQ(\omega)$ grows like $\sqrt{X}$. In \cite{SexticShapes}, we consider such $K$ with a fixed $K_2$ and study their ``$K_2$-shapes'', i.e.\ we take the Minkowski lattice attached to $K$ and project it onto the orthogonal complement of $K_2$ (rather than just the orthogonal complement of $\QQ$). We thus obtain rank $4$ lattices. We show that, when $K_2=\QQ(\omega)$, the shapes live in a one-dimensional space parametrized by a ratio $r_K$ with natural measure $\dfrac{dr_K}{r_K}$, that the number of fields with shape in a compact continuity set grows like $\sqrt{X}$, and that the full space has infinite measure. We also show that, when $K_2\neq\QQ(\omega)$ (and, for simplicity, when the field $K(\omega)$ has class number $1$), the shapes lie in a one-dimensional space that is now parametrized by an \textit{angle} $\theta_K$. The measure of the full space then has finite measure and we show that the number of fields with shape in any continuity set grows like $\sqrt{X}$.
\end{itemize}

These results suggest that the study of shapes could be a fruitful avenue to understanding log terms in asymptotics of counting number fields. In particular, it would be interesting to produce heuristics for these counting functions that rely on understanding what kind of space the given shapes lie in.

\section{Shapes and trace-zero forms of number fields}
In this brief section, we collect the definitions and notation relevant to the discussion of shapes and trace-zero forms of a number field. Some justification for these definitions is provided in the introduction so this section will be pretty dry.

Let $K$ be a degree $n$ number field and let $\sigma_1,\dots,\sigma_n$ be the $n$ distinct embeddings of $K$ into $\CC$. We may collect these together into what we refer to as the \textit{Minkowski embedding of $K$}:
\[
	\funcdef{j}{K}{\CC^n}{\alpha}{(\sigma_1(\alpha),\dots,\sigma_n(\alpha)).}
\]
Let $K_\RR:=\im(j)\otimes_\QQ\RR\subseteq\CC^n$. The restriction of the standard inner product $\langle\cdot\ ,\ \cdot\rangle$ on $\CC^n$ to $K_\RR$ turns $K_\RR$ into a real Euclidean space (see e.g.\ \cite[\S I.5]{Neukirch} for details). We refer to this Euclidean space as the \textit{Minkowski space of $K$} and call its inner product the \textit{Minkowski inner product}. Given any order $R$ in $K$ (or even an ideal in an order), we obtain a full rank lattice $j(R)$ in $K_\RR$. The \textit{shape of $R$} is the equivalence class (under scaling, rotations, and reflections) of the projection of $j(R)$ onto the orthogonal complement of $j(1)$ in $K_\RR$. The \textit{shape of $K$} is the shape of the maximal order $\mc{O}_K$. Concretely, we introduce the ``perp'' map $R\rightarrow R$ defined by
\begin{equation}\label{def:perpmap}
	\alpha^\perp:=n\alpha-\tr(\alpha),
\end{equation}
where $\tr(\alpha)$ is the usual trace map $\tr:K\rightarrow\QQ$. Then, one may verify that $j(\alpha^\perp)$ is $n$ times the orthogonal projection of $j(\alpha)$ onto the orthogonal complement of $j(1)$: the key points being that
\[
	\langle j(1),j(\alpha)\rangle=\tr(\alpha)\quad\text{and}\quad\langle j(1),j(1)\rangle=n.
\]
We let $R^\perp$ be the image of $R$ under this perp map. Then, the shape of $R$ is concretely the equivalence class under scaling, rotations, and reflections of $j(R^\perp)$. Thus, if we begin with a $\ZZ$-basis of $R$, say $1,\alpha_1,\dots,\alpha_{n-1}$, the images of $\alpha_1^\perp,\dots,\alpha_{n-1}^\perp$ under $j$ form a $\ZZ$-basis of $j(R^\perp)$ of $R$. The knowledge of the Gram matrix of the Minkowski inner product with respect to an integral basis of $R$ will then give us, by bilinearity, a Gram matrix representing the shape. The restriction of the Minkowski inner product to $R$ is what we denote $M_R$ and call the \textit{Minkowski form of $R$}.\footnote{Technically, we are composing the Minkowski inner product with the restriction of $j$ to $R$.} Although it does not take values in $\ZZ$, its values are algebraic integers. We denote by $M_R^\perp$ the induced pairing on $R^\perp$. The scaling by $n$ present in the definition of the perp map is chosen so that $M_R^\perp$ again takes integral values. We may sometimes use $M_R^\perp$ to denote the positive-definite binary quadratic form associated to the pairing. We let $M_K$ and $M_K^\perp$ denote the forms for $R=\mc{O}_K$. When $n=3$, $M_K^\perp$ is a positive definite \textit{binary} quadratic form $Q(x,y)=rx^2+sxy+ty^2$ and thus has a point $z(Q)=a+bi$ associated to it in the upper-half plane $\mf{H}=\{x+iy\in\CC:y>0\}$.
\begin{lemma}\label{lem:zQ}
	For a positive definite $Q(x,y)=rx^2+sxy+ty^2$, we have that $z(Q)=a+bi$ where
	\begin{equation}\label{eqn:xycoords}
		a=\dfrac{s}{2r}\quad\text{and}\quad b=\dfrac{\sqrt{4rt-s^2}}{2r}.
	\end{equation}
\end{lemma}
\begin{proof}
	We can think of $Q(x,y)$ as being the quadratic form giving the norm squared of a vector $xv_1+yv_2$ in a lattice spanned by two linearly independent vectors $v_1,v_2\in\CC\cong\RR^2$. We also identify the upper-half plane with the set of lattices in $\CC$ up to rotations, reflections, and scaling as usual, i.e.\ we identify the above lattice with $v_2/v_1$ (after a possible reflection, we may assume $v_2/v_1$ has positive imaginary part). The quadratic form does not change under orthogonal transformations, so we may assume $v_1$ lies on the positive real axis, in which case $v_1=(\sqrt{r},0)$. Then, under the correspondence with $\mf{H}$, we must have that $v_2=(\sqrt{r}a,\sqrt{r}b)$. Taking inner products yields the result.
\end{proof}

We can do something similar by replacing the Minkowski inner product with the so-called trace form. The \textit{trace form of $R$} is
\[
	\funcdef{T_R}{R\times R}{\ZZ}{(\alpha,\beta)}{\tr(\alpha\beta).}
\]
This amounts to replacing the standard inner product on $\CC^n$ with the dot product, i.e.\ we omit taking the complex conjugate of one of the vectors. The \textit{trace-zero form of $R$} is the induced pairing $T_R^\perp$ on $R^\perp$. We call this the ``trace-zero'' form since the orthogonal complement of $j(1)$ in $K_\RR$ is indeed the ($\RR$-span of the image under $j$ of the) space of elements of trace $0$ in $K$. As with $M_R^\perp$, we may also denote by $T_R^\perp$ the binary quadratic form associated to  the pairing. Since the form has $\ZZ$-coefficients, we may also consider the \textit{primitive trace-zero form} $T_R^{\perp\prime}$, which is the associated primitive binary quadratic form (i.e.\ the one where we have divided by the greatest common divisor of the coefficients). Again, we use the subscript $K$ for the case $R=\mc{O}_K$.

\section{The Levi--Delone--Fadeev correspondence and shapes of cubic fields}\label{sec:DF}

The Levi--Delone--Fadeev \cite{DF,Levi} correspondence provides a very useful bijection between isomorphism classes of cubic rings (that is (commutative, unital) rings $R$ that are isomorphic to $\ZZ^3$ as $\ZZ$-modules) and $\gl_2(\ZZ)$-equivalence classes of binary cubic forms. We collect in this section some formulas and features of this correspondence that will be of use to us in subsequent sections. We refer to \cite{Manjul-Arul-Jacob}, especially \S2, for more/other details and unreferenced claims (another modern reference is \cite[\S4]{GGS}).

Let $F(x,y)=ax^3+bx^2y+cxy^2+dy^3$ be a binary cubic form with coefficients in $\ZZ$. Associated to it, we have a cubic ring $R_F$ with a basis $1,\alpha,\beta$ such that
\begin{align}
	\alpha\beta&=-ad,\label{eqn:alphabeta}\\
	\alpha^2&=-ac-b\alpha+a\beta\label{eqn:alpha2},\\
	\beta^2&=-bd-d\alpha+c\beta.
\end{align}
The discriminant of $R_F$ equals that of $F$ and is given by
\begin{equation}\label{eqn:disc}
	b^2c^2-4ac^3-4b^3d-27a^2d^2+18abcd.
\end{equation}
Let $\tr:R\rightarrow\ZZ$ be the trace map (i.e.\ the map that sends $\gamma\in R$ to the trace of the $\ZZ$-linear multiplication-by-$\gamma$ map on $R$).
\begin{lemma}\label{lem:minpolys}
	The elements $\alpha$ and $\beta$ satisfy
	\[
		f_\alpha(x)=a^2F\left(\dfrac{x}{a},1\right)\quad\text{and}\quad f_\beta(x)=\dfrac{x^3}{d}\cdot F\left(\dfrac{-d}{x},1\right)
	\]
	respectively. In particular,
	\begin{align}
		\tr(\alpha)&=-b\label{eqn:tralpha}\\
		\tr(\beta)&=c.\label{eqn:trbeta}
	\end{align}
\end{lemma}
\begin{proof}
	Using \eqref{eqn:alphabeta} and \eqref{eqn:alpha2}, one sees that
	\begin{align*}
		\alpha^3=\alpha\cdot\alpha^2&=-ac\alpha-b\alpha^2+a\alpha\beta\\
								&=-ac\alpha-b\alpha^2-a^2d.
	\end{align*}
	A similar calculation works for $\beta$.
\end{proof}

The action of $\gl_2(\RR)$ on $F(x,y)$ compatible with the correspondence is the so-called \textit{twisted action}
\[
	g\cdot F(x,y):=\dfrac{1}{\det g}F((x,y)g),
\]
where $(x,y)g$ is the linear change of variables given by the vector-matrix multiplication. Then, $\gl_2(\ZZ)$-equivalence classes of forms correspond to isomorphism classes of rings. The forms themselves correspond to pairs consisting of a cubic ring $R$ and a basis of $R/\ZZ$, namely $F$ corresponds to $(R_F,(\alpha+\ZZ,\beta+\ZZ))$.

We have the following connections between $F$ and ring-theoretic properties of $R_F$. Recall that a cubic ring $R$ is said to be \textit{maximal} at a prime $p$ if there is no cubic ring $R^\prime$ such that $R^\prime\supsetneq R$ and $p\mid[R^\prime:R]$.
\begin{proposition}\label{prop:domainmaximal}
	The cubic ring $R_F$ is an integral domain if and only if $F(x,y)$ is irreducible over $\QQ$. It is not maximal at $p$ if and only if $p$ divides all the coefficients of $F$ or there is an element $g\in\gl_2(\ZZ)$ such that $p^2$ divides the $x^3$-coefficient of $g\cdot F$ and $p$ divides the $x^2y$-coefficient.
\end{proposition}
\begin{proof}
	This is \cite[Proposition~12]{Manjul-Arul-Jacob} and the comments after \cite[Lemma~13]{Manjul-Arul-Jacob}.
\end{proof}

Suppose $R_F$ is an order in a complex cubic number field $K$; in particular, $F(x,y)$ is irreducible. Let $\eta$ be a root of $F(x,1)$ in $K$. Using Lemma~\ref{lem:minpolys}, we get that
\[
	\alpha=a\eta\quad\text{and}\quad\beta=-d/\eta.
\]
We now wish to say something about the \textit{shape} of $R_F$. Let $\sigma$ be the real embedding of $K$ and let $\tau, \ol{\tau}$ be the pair of its non-real embeddings. Let $\theta=\sigma(\eta)$, $\xi=\tau(\eta)$, and
let $j:K\hookrightarrow\CC^3$ be the Minkowski embedding $\mu\mapsto(\sigma(\mu),\tau(\mu),\ol{\tau}(\mu))$. Then,
\begin{align*}
	j(\eta)&=(\theta,\xi,\ol{\xi}),\\
	j(\alpha)&=(a\theta,a\xi,a\ol{\xi}),\\
	j(\beta)&=(-d/\theta,-d/\xi,-d/\ol{\xi}).
\end{align*}
As a first step towards determining the shape of $R_F$, we will determine the Gram matrix of the Minkowski form of $R_F$ with respect to the basis $1,\alpha,\beta$. By construction, the entries of the Gram matrix will lie in the maximal real subfield of the Galois closure of $K$ and so, following \cite[\S9]{terr}, we will use the $\QQ$-basis $\theta^{-1},1,\theta$ for the field $\QQ(\theta)\subseteq\RR$.
\begin{lemma}
	The minimal polynomial of $\xi$ over $\QQ(\theta)$ is
	\[
		x^2-t_\xi x+n_\xi,
	\]
	where
	\begin{align*}
		t_\xi&=-\frac{b}{a}-\theta\\
		n_\xi&=-\frac{d}{a}\theta^{-1}.
	\end{align*}
\end{lemma}
\begin{proof}
	Note that over $\QQ(\theta)$, we have
	\begin{align*}
		a(x-\theta)(x^2-t_\xi x+n_\xi)&=ax^3+(-a)(t_\xi+\theta)x^2+a(n_\xi+\theta)x+(-a)\theta n_\xi\\
		&=ax^3+bx^2+cx+d\\
		&=F(x,1).
	\end{align*}
\end{proof}
\begin{proposition}
	The Gram matrix of $M_{R_F}$ with respect to the basis $1,\alpha,\beta$ is
	\[
		\begin{pmatrix}
			3&-b&c\\
			-b&-a(3d\theta^{-1}+c+b\theta)&-(bd\theta^{-1}+bc+ac\theta)\\
			c&-(bd\theta^{-1}+bc+ac\theta)&-d(c\theta^{-1}+b+3a\theta)
		\end{pmatrix}.
	\]
\end{proposition}
\begin{proof}
	For any $\gamma\in K$, $\langle j(1),j(\gamma)\rangle=\tr(\gamma)$, so the first row (and column) follow from equations \eqref{eqn:tralpha} and \eqref{eqn:trbeta}. The rest follows from the previous lemma and the fact that each of $\theta,\xi,$ and $\ol{\xi}$ satisfy $F(x,1)$. For instance, one may begin with
	\begin{align*}
		\langle j(\alpha),j(\beta)\rangle&=-ad\left(1+\frac{\xi}{\ol{\xi}}+\frac{\ol{\xi}}{\xi}\right)\\
			&=-ad\left(1+\frac{\xi^2+\ol{\xi}^2}{\xi\ol{\xi}}\right)\\
			&=-ad\left(1+\frac{(\xi+\ol{\xi})^2-2\xi\ol{\xi}}{\xi\ol{\xi}}\right)\\
			&=-ad\left(1+\frac{t_\xi^2-2n_\xi}{n_\xi}\right).
	\end{align*}
\end{proof}

\begin{proposition}\label{prop:general_shape}
	The Gram matrix of $M_{R_F}^\perp$ with respect to the basis $\alpha^\perp,\beta^\perp$ is
	\begin{equation}\label{eqn:MGram}
		\begin{pmatrix}
			-3(9ad\theta^{-1}+(b^2+3ac)+3ab\theta)&-3(3bd\theta^{-1}+2bc+3ac\theta)\\
			-3(3bd\theta^{-1}+2bc+3ac\theta)&-3(3cd\theta^{-1}+(c^2+3bd)+9ad\theta)
		\end{pmatrix}.
	\end{equation}
\end{proposition}
\begin{proof}
	This follows from the previous proposition and bilinearity, using the formula \eqref{def:perpmap} for the perp map.
\end{proof}
We also denote by $M_{F}^\perp$ the binary quadratic form associated to this Gram matrix. The map $F\mapsto M_{F}^\perp$ is $\gl_2(\RR)$-equivariant, where the action of $g\in\gl_2(\RR)$ on a binary quadratic form $Q(x,y)$ is $(g\cdot Q)(x,y)=Q((x,y)g)$. We denote by $\sh(R_F)$ the point $z(M_F^\perp)\in\mf{H}$.

The \textit{Hessian} of $F$ is the binary quadratic form (\cite[eq.~(2)]{Manjul-Ari})
\begin{equation}\label{eqn:Hessian}
	H_F(x,y):=(b^2-3ac)x^2+(bc-9ad)xy+(c^2-3bd)y^2.
\end{equation}
It represents the trace-zero form of $R_F$ (\cite[Proposition~12]{Manjul-Ari}) (up to scaling). The association $F\mapsto H_F$ is $\gl_2(\RR)$-equivariant and we have that (\cite[Proposition~11]{Manjul-Ari})
\begin{equation}
	\Delta(H_F)=-3\Delta(F).
\end{equation}

Note however that the Gram matrix of the trace-zero form of $R_F$ with respect to $\alpha^\perp,\beta^\perp$ is
\begin{equation}\label{eqn:TGram}
	\begin{pmatrix}
		6(b^2-3ac)	&3(bc-9ad)\\
		3(bc-9ad)	&6(c^2-3bd)
	\end{pmatrix}.
\end{equation}
This is the Gram matrix of $6H_F$, accordingly we let $T^\perp_{F}:=6H_F$.

For technical reasons, in \S\ref{sec:equidistribution}, we will need to work with \textit{oriented} cubic rings as in \cite[p.\ 55]{Manjul-Ari}. An \textit{oriented cubic ring} is a pair $(R,\delta)$ where $R$ is a cubic ring and $\delta$ is an isomorphism $\wedge^3R\cong\ZZ$. This orientation has the effect of fixing an ordered basis of $R/\ZZ$ and as such the shape of an oriented ring is only determined up to $\SL_2(\ZZ)$-equivalence, and therefore lives in $\SL_2(\ZZ)\backslash\mf{H}$. We will be able to translate results about shapes of oriented cubic rings into ones about plain old cubic rings.

\section{The Tame-Wild Dichotomy}\label{sec:TameVsWild}
In this section, we prove Theorem~\ref{thm:TameVsWild}, which provides an interpretation for the discriminant of the primitive trace-zero form in terms of the tame versus wild ramification of $3$ in the cubic field $K$.

We begin by remarking that if $3$ is ramified in $K$, then $\ord_3(\Delta(K))=1,3,4,$ or $5$. Indeed, letting $\mc{D}(K)$ denote the different of $K$, if $3$ factors as $\mf{p}_1^2\mf{p}_2$, then $\mf{p}_1$ is tamely ramified, so that $\ord_{\mf{p}_1}(\mc{D}(K))=1$, and hence $\ord_3(\Delta(K))=1$. Otherwise, $3=\mf{p}^3$ is wildly ramified, so that $3\leq\ord_\mf{p}(\mc{D}(K))\leq5$ by the standard inequality (from the standard reference \cite[{\S}III.6, Proposition 13, and following remark]{Serre}). Since the quadratic resolvent field of $K$ is $\QQ(\sqrt{\Delta(K)})$, the assumption that $3$ divides the latter's discriminant is equivalent to $\ord_3(\Delta(K))$ being odd. Thus, we wish to prove the following result.
\begin{proposition}\label{prop:ord3TZP}
	If $K$ is a cubic field such that the discriminant of its quadratic resolvent field is divisible by $3$, then
	\[
		\ord_3\Delta\!\left(\TZP\right)=\begin{cases}
				0 & \text{if }\ord_3\Delta(K)=3\text{ or }5,\\
				2 & \text{if }\ord_3\Delta(K)=1.
			\end{cases}
	\]
\end{proposition}

We will require the following lemma.

\begin{lemma}
	Suppose $F(x,y)=ax^3+bx^2y+cxy^2+dy^3$ is a binary cubic form corresponding to a maximal cubic ring.
	\begin{enumerate}
		\item If $F$ is a cube modulo $3$, then it is $\gl_2(\ZZ)$-equivalent to a form with $a\equiv b\equiv c\equiv0\mod{3}$ and $d\not\equiv0\mod{3}$.
		\item If $F$ factors as a $G_1G_2^2$ modulo $3$, where $G_1$ and $G_2$ are linear forms that aren't constant multiples of each other, then it is $\gl_2(\ZZ)$-equivalent to a form with $a\equiv b\equiv d\equiv0\mod{3}$ and $c\not\equiv0\mod{3}$.
	\end{enumerate}
\end{lemma}
\begin{proof}
	If $F(x,y)\equiv(mx+ny)^3\mod{3}$, and $n\equiv0\mod{3}$, then switching $x$ and $y$ does the trick. Otherwise, the change of variable $x^\prime=x$, $y^\prime=-n^{-1}mx+y$ (where $n^{-1}$ is an inverse of $n$ modulo $3$) moves $F$ to a form congruent to $ny^3\mod{3}$. If this new coefficient of $y^3$ were divisible by $3$, then $F$ would correspond to a non-maximal ring by Proposition~\ref{prop:domainmaximal}.
	
	If $F\equiv (m_1x+n_1y)(m_2x+n_2y)^2\mod{3}$, a similar change of variables allows us to assume that $m_2\equiv n_1\equiv0\mod{3}$. Similarly, $c\not\equiv0\mod{3}$ since $F$ corresponds to a maximal ring.
\end{proof}

\begin{proof}[Proof of Proposition~\ref{prop:ord3TZP}]
	Let $F(x,y)=ax^3+bx^2y+cxy^2+dy^3$ be a binary cubic form corresponding to the ring of integers of $K$. Write the Hessian of $F$ as $H_F(x,y)=rx^2+sxy+ty^2$ and recall that it is an integer multiple of $\TZP$. Suppose first that $3$ is wild in $K$. Then, by \cite[Lemma 11]{DH}, $F(x,y)$ is a cube modulo $3$, and so, by the previous lemma, we may assume that $3$ divides each of $a, b,$ and $c$, but not $d$.	By \eqref{eqn:Hessian}, $3^2\mid\gcd(r,s,t)$, so that
	\[
		\ord_3\Delta\!\left(\TZP\right)\leq\ord_3(\Delta(H_F))-4=\ord_3(\Delta(K))-3.
	\]
	This gives the desired result when $\ord_3(\Delta(K))=3$. When this valuation is $5$, we note, by \eqref{eqn:disc}, that $\Delta(F)\equiv-4b^3d\mod{3^4}$, so that $3^2\mid b$. Then,
	\[
		\Delta(F)\equiv-4ac^3\mod{3^5}.
	\]
	 Since $F$ corresponds to a maximal ring, $3^2\nmid a$, so that $3^2\mid c$. This now shows that $3^3\mid\gcd(r,s,t)$, so that
	 \[
		\ord_3\Delta\!\left(\TZP\right)\leq\ord_3(\Delta(H_F))-6=\ord_3(\Delta(K))-5=0,
	\]
	as desired.
	
	Now, suppose $3$ is tamely ramified in $K$. Again by \cite[Lemma 11]{DH} and the previous lemma, we may assume that now $a\equiv b\equiv d\equiv 0\mod{3}$, but $c\not\equiv0\mod{3}$. Then, $3\nmid t$. Thus,
	\[
		\ord_3\Delta\!\left(\TZP\right)=\ord_3(\Delta(H_F))=\ord_3(\Delta(K))+1=2.
	\]
\end{proof}

\section{The shape as a complete invariant}\label{sec:complete_invariant}
In this section, we prove that the shape is a complete invariant in the family of complex cubic fields (Theorem~\ref{thm:complete_invariant} above). More precisely, as stated above, we prove the stronger fact that the shape of a given (isomorphism class of a) complex cubic field is distinct from the shape of any other cubic field. This comes down to an (ir)rationality argument, and is reminiscent of how the discriminant is a complete invariant of quadratic fields: the quadratic field of discriminant $\Delta$ \textit{is} the field obtained by adjoining a square root of $\Delta$ to $\QQ$. Similarly, we will see that the complex cubic field of shape $x+iy\in\mf{H}$ \textit{is} (isomorphic to) the field obtained by adjoining $x$ (or, in the pure cubic case, $y$) to $\QQ$.

First, note that if $K$ is a real cubic field, then we have an equality of binary quadratic forms $M_K^\perp=T_K^\perp$ so that $M_K^\perp$ has coefficients in $\ZZ$. By \eqref{eqn:xycoords}, the $x$-coordinate of $\sh(K)$ is in $\QQ$ and its $y$-coordinate is in a quadratic extension of $\QQ$. Either way, these coordinates are not irrational cubic numbers. Theorem~\ref{thm:complete_invariant} was proved by the author when $K$ is a pure cubic field, in \cite[Theorem~B]{PureCubicShapes}, essentially by showing that (the image of the real embedding of) $K$ is $\QQ(\Im(\sh(K))$. As noted above, $\Re(\sh(K))$ lies in the image of $K$ under its real embedding. It thus suffices to prove the following proposition.

\begin{proposition}\label{prop:irrationality}
	If $K$ is a non-pure, complex cubic field, then $\Re(\sh(K))$ is irrational.
\end{proposition}
\begin{proof}
	Let $F(x,y)=ax^3+bx^2y+cxy^2+dy^3$ be a binary cubic form with integer coefficients corresponding to the ring of integers of $K$. The formula for the shape in Proposition~\ref{prop:general_shape} (and \eqref{eqn:xycoords}) shows that
	\begin{equation}\label{eqn:ReShK}
		\Re(\sh(K))=\frac{3bd\theta^{-1}+2bc+3ac\theta}{9ad\theta^{-1}+(b^2+3ac)+3ab\theta}\in\QQ(\theta).
	\end{equation}
	Introduce three unknown rational numbers $c_1,c_2,$ and $c_3$ such that $c_1\theta^{-1}+c_2+c_3\theta$ is the inverse of the denominator $9ad\theta^{-1}+(b^2+3ac)+3ab\theta$.
	Assume for the sake of contradiction that $\Re(\sh(K))=\rho\in\QQ$. Then, the two equations
		\begin{align*}
			1&=(9ad\theta^{-1}+(b^2+3ac)+3ab\theta)\cdot(c_1\theta^{-1}+c_2+c_3\theta)\\
			\rho&=(3bd\theta^{-1}+2bc+3ac\theta)\cdot(c_1\theta^{-1}+c_2+c_3\theta)
		\end{align*}
	give a system of six linear equations over $\QQ$ in the three unknowns $c_1,c_2,$ and $c_3$ represented by the following matrix
	\[
		M=\begin{pmatrix}
b^{2} - 6 a c & 9 a d & -3 b d & 0 \\
-6 a b & b^{2} + 3 a c & 9 a d -3 b c & 1 \\
-9 a^{2} & 3 a b & -2 b^{2} + 3 a c & 0 \\
- b c & 3 b d & -3 c d & 0 \\
-3 b^{2} + 3 a c & 2 b c & -3 c^{2} + 3 b d & \rho \\
-3 a b & 3 a c & - b c & 0
\end{pmatrix}.
	\]
	We will show that this system has no solutions under the hypotheses of this proposition. We note the following consequences of the hypotheses and the rows of $M$. First, since $F(x,y)$ corresponds to an order in an $S_3$-cubic field, it is irreducible, so both $a$ and $d$ are non-zero. Also, $b$ and $c$ can't both be zero: if they are, then $K=\QQ(\sqrt[3]{-d/a})$, which is a pure cubic. If $b=0$, then the fourth row of $M$ implies that $c_3=0$, which in turn implies, from row three, that $c_1=0$. The first row then implies that $c_2=0$, so that the inverse of the denominator in \eqref{eqn:ReShK} is zero: clearly a contradiction. A similar argument shows that $c$ can't be zero either. Finally, the expression $b^2-3ac$ is non-zero: otherwise, the Hessian of $F(x,y)$ has square discriminant so that $F(x,y)$ corresponds to a pure cubic field \cite[Lemma 33]{Manjul-Ari}. Given these non-vanishing statements, we can use the matrix
	\[
E=\begin{pmatrix}
\frac{b}{b^{2} - 3 a c} & 0 & 0 & \frac{-3 a}{b^{2} - 3 a c} & 0 & 0 \\
\frac{b(b c - 9 a d)}{-3d( b^{2} - 3 a c)} & 1 & 0 & \frac{\frac{1}{3} b^{3} - 2 a b c + 9 a^{2} d}{-d( b^{2} - 3 a c)} & 0 & -1 \\
\frac{9 a^{2} b}{b^{2} - 3 a c} & 0 & b & \frac{-27 a^{3}}{b^{2} - 3 a c} & 0 & 0 \\
\frac{b c}{3d(b^{2} - 3 a c)} & 0 & 0 & \frac{ b^{2} - 6 a c}{3d(b^{2} - 3 a c)} & 0 & 0 \\
3 a b & 0 & -\frac{2}{3} b c & -9 a^{2} & a b & 0 \\
\frac{3 a b}{b^{2} - 3 a c} & 0 & 0 & \frac{-9 a^{2}}{b^{2} - 3 a c} & 0 & 1
\end{pmatrix}
\]
to effect a row reduction of $M$ yielding
\[
EM=\begin{pmatrix}
b & 0 & -3 d & 0 \\
0 & 0 & 0 & 1 \\
0 & 3 a b^{2} & -2 b^{3} + 3 a b c - 27 a^{2} d & 0 \\
0 & b & -2 c & 0 \\
0 & 0 & \frac{4}{3} b^{3} c - 5 a b c^{2} - 6 a b^{2} d + 27 a^{2} c d & a b \rho \\
0 & 3 a c & - b c - 9 a d & 0
\end{pmatrix}
\]
whose second row shows this system is inconsistent, as desired.
\end{proof}
Note that the proof did not require that $F(x,y)$ correspond to a maximal order, therefore this irrationality result still holds for orders in non-pure complex cubic fields.

\subsection{The shapes avoid the boundary}\label{sec:avoidboundary}
In this brief section, we will prove Theorem~\ref{thm:avoidboundaries} as a corollary of Proposition~\ref{prop:irrationality}. The standard Gauss fundamental domain for $\gl_2(\ZZ)\backslash\mf{H}$ is
\[
	\mc{G}=\left\{x+iy:0\leq x\leq\frac{1}{2},x^2+y^2\geq1\right\}.
\]
The boundary of this set consists of two vertical rays
\[
	\{iy:y\geq1\}\quad\text{and}\quad\left\{\frac{1+iy\sqrt{3}}{2}:y\geq1\right\}
\]
and the circular arc
\[
	\{\cos(\theta)+i\sin(\theta):\pi/3\leq\theta\leq\pi/2\}.
\]
If $R$ is an order in a non-pure complex cubic field, then Proposition~\ref{prop:irrationality} above shows that its shape cannot lie on the two vertical boundary components. To prove Theorem~\ref{thm:avoidboundaries}, it therefore suffices to show that the shape cannot lie on the circular arc of the boundary. The matrix
\[
	\vect{0&1\\-1&1}\in\SL_2(\ZZ)
\]
maps the unit semi-circle $\{\cos(\theta)+i\sin(\theta):0<\theta<\pi\}$ to the vertical ray $\left\{\frac{1+iy}{2}:y>0\right\}$ by fractional linear transformation. Thus, if $R$ corresponds to some binary cubic form $F$ such that $M_F^\perp$ gives a point on the circular arc of the boundary, then $F$ is $\SL_2(\ZZ)$-equivalent to a form $F^\prime$ with $M_{F^\prime}^\perp$ giving a point whose real part is $1/2$. This latter option is impossible by Proposition~\ref{prop:irrationality}, so we have proved Theorem~\ref{thm:avoidboundaries}.


\section{Indefinite binary quadratic forms and geodesics on the modular surface}\label{sec:geodesics}

The material in this section is only needed for \S\ref{sec:equidistribution}. We collect some facts and constructions regarding the correspondence between indefinite integral binary quadratic forms and certain geodesics on the modular surface. We refer the reader to \cite{SarnakReciprocal} and \cite[Ch.~9]{EinsiedlerWard} for more details.

Let $Q(x,y)=rx^2+sxy+ty^2$ be a real binary quadratic form and let $D=s^2-4rt$ be its discriminant. There is an action of $\gl_2(\RR)$ on binary quadratic forms given by $(g\cdot Q)(x,y)=Q((x,y)g)$. The \textit{connected component of the identity of the orthogonal similitude group of $Q$} is
\[
	\GO_Q^0(\RR)=\left\{g\in\gl_2(\RR):g\cdot Q=c_gQ,c_g>0,\det(g)>0\right\}.
\]
Let $Q_0(x,y)=xy$. A straightforward calculation shows that $\GO^0_{Q_0}(\RR)$ is the group of diagonal matrices with positive determinant. There is a group isomorphism
\[
	\begin{tikzcd}[row sep=0.2em]
		\RR_{>0}\times\RR^\times\ar[r, "\sim"]&\GO^0_{Q_0}(\RR)\\
		(\lambda,\alpha)\ar[r, mapsto]&g(\lambda,\alpha),
	\end{tikzcd}
\]
where
\[
	g(\lambda,\alpha):=\vect{\lambda\alpha^{-1}\\&\lambda\alpha}.
\]
Suppose $Q(x,y)=rx^2+sxy+ty^2$ is indefinite (i.e.\ $D>0$) with $t\neq0$. Let
\[
	\theta_\pm:=\dfrac{s\pm\sqrt{D}}{2t}
\]
and let
\[
	P=\sqrt{t}\vect{\theta_+&\theta_-\\1&1}.
\]
Then, $P\cdot Q_0=Q$ so that $\GO^0_Q(\RR)=P\GO^0_{Q_0}(\RR)P^{-1}$. Suppose now that $Q$ is an indefinite \textit{integral} binary quadratic form of non-square determinant (in particular $t\neq0$). Let
\[
	\GO^0_Q(\ZZ):=\GO^0_Q(\RR)\cap\gl_2(\ZZ)=\GO^0_Q(\RR)\cap\SL_2(\ZZ)=:\SO_Q(\ZZ).
\]
The set of integer solutions $(U,W)\in\ZZ^2$ to the Pellian equation $u^2-Dw^2=4$ forms a group isomorphic to $\{\pm1\}\times\ZZ$. There is an injective homomorphism of this group into $\QQ(\sqrt{D})^\times$ sending $(U,W)$ to $\frac{1}{2}(U+W\sqrt{D})$. There are then four elements $\epsilon_0=\frac{1}{2}(U_0+W_0\sqrt{D})$ in the image of this homomorphism such that the image is $\{\pm\epsilon_0^m:m\in\ZZ\}$. We denote by $\epsilon_0$ the unique such element that is $>1$. Then, $\langle\pm1,\epsilon_0\rangle$ is isomorphic to $\SO_Q(\ZZ)$ via
\[
	\frac{1}{2}(U+W\sqrt{D})\mapsto M(U,W):=\vect{\dfrac{1}{2}(U-sW)&rW\\-tW&\dfrac{1}{2}(U+sW)}.
\]
It can be verified that
\[
	P\vect{\epsilon_0^{-1}\\&\epsilon_0}P^{-1}=M(U_0,W_0),
\]
so that
\[
	\SO_Q(\ZZ)=P\left\{\pm\vect{\epsilon_0^{-1}\\&\epsilon_0}^m:m\in\ZZ\right\}P^{-1}.
\]

If $Q(x,y)$ is any (real) indefinite binary quadratic form, it has two ``roots'' $\rho_{\pm}\in\PP^1(\RR)$, i.e. ${\rho_\pm=[\rho_{\pm,x}:\rho_{\pm,y}]\in\PP^1(\RR)}$ such that
\[
	Q(x,y)=r(\rho_{+,y}x-\rho_{+,x}y)(\rho_{-,y}x-\rho_{i,x}y).
\]
This allows one to associate to $Q$ a geodesic $\gamma_Q$ in $\mf{H}$. Namely, if neither of $\rho_\pm$ is infinite, $\gamma_Q$ is the semicircle whose diameter is the line connecting $\rho_+$ and $\rho_-$; otherwise, $\gamma_Q$ is a vertical line connecting $\infty$ to whichever of $\rho_\pm$ is not infinite. Note that $\gamma_{Q_0}$ is the positive imaginary axis $\{iy:y>0\}$. There is an action of $\gl_2(\RR)$ on $\ol{\mf{H}}:=\mf{H}\cup\PP^1(\RR)$ by fractional linear transformation given by
\[
	g\cdot z=\vect{a&b\\c&d}\cdot z=\begin{cases}
							\dfrac{az+b}{cz+d} &\det g>0\\
							\dfrac{a\ol{z}+b}{c\ol{z}+d} &\det g<0.\\
						\end{cases}
\]
Letting $g\ast z:=(g^{-1})^T\cdot z$ for $g\in\gl_2(\RR)$, we have that the association $Q\mapsto\gamma_Q$ is $\gl_2(\RR)$-equivariant for this action $\ast$ on $\ol{\mf{H}}$. Furthermore, if $Q(x,y)$ is any (real) definite binary quadratic form, then the association $Q\mapsto z(Q)$ (from Lemma~\ref{lem:zQ}) is also $\gl_2(\RR)$-equivariant for the action $\ast$ on $\mf{H}$. We denote by $[Q]$ (resp.\ $[Q]_1$) the equivalence class of $Q$ under $\gl_2(\ZZ)$ (resp.\ $\SL_2(\ZZ)$), similarly for $[z(Q)]\in\gl_2(\ZZ)\backslash\mf{H}$ (resp.\ $[z(Q)]_1\in\SL_2(\ZZ)\backslash\mf{H}$). We will identify $\gl_2(\ZZ)\backslash\mf{H}$ (resp.\ $\SL_2(\ZZ)\backslash\mf{H}$) with the Gauss fundamental domain $\mc{G}\subseteq\mf{H}$ given by
\begin{equation}\label{eqn:G}
	\mc{G}=\{x+iy:0\leq x\leq\frac{1}{2},x^2+y^2\geq1\}
\end{equation}
and
\begin{equation}\label{eqn:G1}
	\mc{G}_1=\{x+iy:|x|\leq\frac{1}{2},x^2+y^2\geq1\},
\end{equation}
respectively.
We denote by $[\gamma_Q]\subseteq\gl_2(\ZZ)\backslash\mf{H}$ (resp.\ $[\gamma_Q]_1\subseteq\SL_2(\ZZ)\backslash\mf{H}$) the set of points in $\mc{G}$ that are $\gl_2(\ZZ)$-equivalent (resp.\ in $\mc{G}_1$ that are $\SL_2(\ZZ)$-equivalent) to points on $\gamma_Q$. This only depends on $[Q]$ (resp.\ $[Q]_1$).
Then, $[\gamma_{Q_0}]=[\gamma_{Q_0}]_1=\{iy:y\geq1\}$. The matrix
\[
	g_\alpha=\vect{\alpha^{-1}\\&\alpha}, \alpha\in\RR_{>0}
\]
flows the point $i\in\gamma_{Q_0}$ along the geodesic to $g_\alpha\ast i=\alpha^2i$. Thus, $Pg_\alpha P^{-1}$ flows $P\ast i$ along the geodesic $\gamma_Q$.

Suppose from now on that $Q$ is an indefinite integral binary quadratic form with non-square discriminant. As $\alpha$ varies on the interval $(\epsilon_0^{-1},1]$, the flow by $Pg_\alpha P^{-1}$ gives an $n_Q$-fold cover of $[\gamma_Q]_1$ for some positive integer $n_Q$.\footnote{This will never be a simple cover when, for instance, $Q$ is reciprocal.} Specifically, let
\[
	\mc{T}:=\left\{\vect{\alpha^{-1}\\&\alpha}:\alpha\geq1\right\},
\]
\[
	\mc{T}_{>\beta}:=\left\{\vect{\alpha^{-1}\\&\alpha}:1\geq\alpha>\beta\right\},
\]
and
\[
	\mc{T}_{Q}:=\mc{T}_{>\epsilon_0}.
\]
Then,
\begin{equation}\label{eqn:piQ}
	\begin{tikzcd}[row sep=0.2em]
		\pi_Q:\mc{T}_Q\ar[r]&{[\gamma_Q]_1}\\
		\phantom{\pi_Q:\ }g_\alpha\ar[r, mapsto]&{[Pg_\alpha\ast i]_1}
	\end{tikzcd}
\end{equation}
is this $n_Q$-fold cover. The hyperbolic measure on $[\gamma_{Q_0}]_1$ is $d\mu_{Q_0}(iy)=d^\times y=\dfrac{dy}{y}$. It is invariant under the action of the diagonal subgroup of $\SL_2(\RR)$. 
If $f:[\gamma_{Q_0}]_1\rightarrow\CC$ is an $L^1$ function, then
\[
	\int_{[\gamma_{Q_0}]_1}f(z)d\mu_{Q_0}(z)=2\int_\mc{T}f(g_\alpha\ast i)d^\times\alpha.
\]
The hyperbolic measure $\mu_Q$ on $[\gamma_Q]_1$ is obtained by pushing $\mu_{Q_0}$ forward using $P$, i.e.\ if $f:[\gamma_{Q}]_1\rightarrow\CC$ is an $L^1$ function, then
\begin{equation}\label{eqn:transfer_integral}
	\int_{[\gamma_Q]_1}f(z)d\mu_Q(z)=\dfrac{2}{n_Q}\int_{\mc{T}_{Q}}f(\pi_Q(g_\alpha))d^\times\alpha.
\end{equation}


\section{Equidistribution on geodesics}\label{sec:equidistribution}
In this section, we prove Theorems~\ref{thm:3TKonMK} and \ref{thm:Equidistribution} on the equidistribution of shapes of complex cubic fields on certain geodesics of the modular surface $\gl_2(\ZZ)\backslash\mf{H}$ building on the correspondences of the previous section, \cite{Manjul-Piper, PiperThesis}, and \cite{Manjul-Ari}.

To begin with, once we realize that the shape of a complex cubic field $K$ lies on the geodesic determined by its trace-zero form, we can easily show it. Indeed, take your favourite complex cubic field, or a straightforward one like $\QQ(\sqrt[3]{2})$, and compute its shape and its trace-zero form. For $\QQ(\sqrt[3]{2})$, this gives
\[
	\sh(\QQ(\sqrt[3]{2}))=i2^{1/3}\quad\text{and}\quad T_{\QQ(\sqrt[3]{2})}^{\perp\prime}(x,y)=xy=Q_0,
\]
respectively. Verify that the shape lies on the corresponding geodesic, as $i2^{1/3}$ lies on the geodesic $\gamma_{Q_0}$ connecting $0$ to $\infty$. Now, note that the formation of the shape, the trace-zero form, and the associated geodesic are all $\gl_2(\RR)$-equivariant. Since the set of (binary cubic forms corresponding to) complex cubic fields (or, more generally, complex cubic orders) is contained in a unique $\gl_2(\RR)$ orbit, the shape of every complex cubic field lies on the geodesic associated to its trace-zero form.

For the matter of equidistribution, we apply the method of \cite{Manjul-Piper, PiperThesis} to the work of \cite{Manjul-Ari}. For technical reasons, we will follow \cite{Manjul-Ari} and use oriented cubic rings in our proofs.

Let $[Q]_1$ be the $\SL_2(\ZZ)$-equivalence class of a primitive indefinite binary quadratic form
\[
	Q(x,y)=rx^2+sxy+ty^2
\]
of nonsquare discriminant $D=s^2-4rt>0$. We may assume that $t>0$. We introduce some relevant objects from \cite{Manjul-Ari}. Let
\[
	Q^\prime(x,y)=tx^2-sxy+ry^2
\]
be the adjoint quadratic form of $Q$ so that
\[
	Q^\prime(x,y)=t(x-\theta_+y)(x-\theta_-y).
\]
Let
\[
	V_\RR=\RR^2,\quad V_\ZZ=\ZZ^2,
\]
\[
	\VQ_\RR=\left\{(x,y)\in V_\RR:\frac{Q^\prime(x,y)}{rt}>0\right\}
\]
and
\[
	\VQ_\ZZ=	\left\{(x,y)\in V_\ZZ:\frac{Q^\prime(x,y)}{rt}>0,sb\equiv rc\mod{3t},sc\equiv tb\mod{3r}\right\}.
\]
For $(b,c)\in\VQ_\RR$, define
\[
	\Delta(b,c):=-\dfrac{Q^\prime(b,c)^2\Delta(Q)}{3r^2t^2}.
\]
If $F(x,y)=ax^3+bx^2y+cxy^2+dy^3$ is a binary cubic form whose Hessian is a positive multiple of $Q$, let $v_F:=(b,c)\in V_\RR$. Following \cite{Manjul-Ari}, we define a \textit{twisted cubic} action $\ast$ of $\GO^0_Q(\RR)$ on $V_\RR$ via
\[
	g\ast\vect{x\\y}=\dfrac{1}{\det g}\ g^3\!\vect{x\\y},
\]
where the latter denotes the usual matrix-vector multiplication. As in \cite{Manjul-Ari}, one can verify that for $g\in\GO^0_Q(\RR)$ and $F$ whose Hessian is a positive multiple of $Q$,
\[
	v_{g\cdot F}=g\ast v_F.
\]

A crucial ingredient to our proof is the following correspondence of Bhargava--Shnidman.
\begin{theorem}[Theorem~21 of \cite{Manjul-Ari}]
	The isomorphism classes of oriented cubic rings whose primitive trace-zero form is $Q$ are in natural bijection with $\SO_Q(\ZZ)$-orbits (with respect to the action $\ast$) of pairs $(b,c)\in\VQ_\ZZ$. Under this bijection,
	$R$ is the ring corresponding to the binary cubic form
	\[
		ax^3+bx^2y+cxy^2+dy^3,
	\]
	where
	\[
		a=\dfrac{sb-rc}{3t}\quad\text{and}\quad d=\dfrac{sc-tb}{3r}.
	\]
	Furthermore,
	\[
		\Delta(R)=\Delta(b,c).
	\]
\end{theorem}

We port over the definition of irreducible to the space $\VQ_\ZZ$ and say that an element $v\in \VQ_\ZZ$ is \textit{irreducible} if its associated binary cubic form is irreducible.
By Proposition~\ref{prop:domainmaximal}, this is equivalent to the associated cubic ring being an integral domain, i.e.\ an order in a cubic field.

In order to link \cite{Manjul-Ari} to \cite{Manjul-Piper}, we relate constructions in the former to the group $\GO^0_Q(\RR)$. Note that $\GO^0_Q(\RR)\cap\gl_2(\ZZ)=\SO_Q(\ZZ)$.
\begin{lemma}
	$\VQ_\RR$ is a $\GO^0_Q(\RR)$-orbit.
\end{lemma}
\begin{proof}
	Note that for $v=(x,y)\in\RR^2$,
	\[
		P^{-1}\cdot\vect{x\\y}=\frac{1}{\sqrt{tD}}\vect{x-\theta_-y\\ -(x-\theta_+y)}.
	\]
	For each of $+$ and $-$, let
	\[
		V^{(Q_0),\pm}_\RR=\{(x,y)\in\RR^2:\pm Q^\prime_0(x,y)=\mp xy>0\}.
	\]
	Let $(x^\prime,y^\prime):=(x-\theta_-y,-(x-\theta_+y))$. Then,
	\begin{align*}
		v=(x,y)\in V_\RR\quad	&\text{if and only if}\quad -\frac{x^\prime y^\prime}{r}>0\\
						&\text{if and only if}\quad P^{-1}\cdot v\in V^{(Q_0),\sgn(r)}_\RR.
	\end{align*}
	The lemma thus comes down to showing that each of $V^{(Q_0),\pm}_\RR$ is a $\GO^0_{Q_0}(\RR)$-orbit.
	
	So, let $v_0=(1,\mp1)\in V^{(Q_0),\pm}_\RR$, then
	\[
		g(\lambda,\alpha)\ast v_0=g(\lambda^3,\alpha^3)v_0=(\lambda^3\alpha^{-3},\mp\lambda^3\alpha^3)\in V^{(Q_0),\pm}_\RR
	\]
	so that $V^{(Q_0),\pm}_\RR$ contains an orbit. On the other hand, if $(x,y)\in V^{(Q_0),\pm}_\RR$, then $(x,y)=g(\lambda,\alpha)\ast v_0$ for $\lambda=\sqrt{\mp xy}$ and $\alpha=\sqrt{\mp y/x}$.
\end{proof}
Using this lemma, and the equivariance of all related constructions, we may define the shape of an element $v\in\VQ_\RR$, as follows. There is some $v_1\in\VQ_\ZZ$ that corresponds to some complex cubic order $R_1$. There is then some $g\in\GO^0_Q(\RR)$ such that $g\ast v_1=v$. Define $\sh(v):=g\ast\sh(R_1)$. This gives a well-defined $\GO^0_Q(\RR)$-equivariant surjective map $\sh:\VQ_\RR\rightarrow[\gamma_Q]_1$. Note that for $g\in\GO^0_Q(\RR)$,
\begin{equation}\label{eqn:discactiononVQ}
	\Delta(g\ast(b,c))=\det(g)^2\Delta(b,c).
\end{equation}
For $v=\vect{x\\y}\in\VQ_\RR$, define its \textit{ratio} to be
\[
	\ratio(v):=\dfrac{x-\theta_-y}{x-\theta_+y}.
\]
Note that $\ratio(-v)=\ratio(v)$.
For $Pg(\lambda,\alpha)P^{-1}\in\GO^0_Q(\RR)$, define
\[
	\ratio(Pg(\lambda,\alpha)P^{-1}):=\alpha^{-2}.
\]
A calculation shows that
\[
	\ratio(Pg(\lambda,\alpha)P^{-1}\cdot v)=\ratio(Pg(\lambda,\alpha)P^{-1})\cdot\ratio(v),
\]
i.e.
\begin{equation}\label{eqn:ratio}
	\ratio(Pg(\lambda,\alpha)P^{-1}\ast v)=\ratio(Pg(\lambda,\alpha)P^{-1})^3\cdot\ratio(v).
\end{equation}
Therefore, every $\SO_Q(\ZZ)$ orbit in $\VQ_\ZZ$ contains a unique element $v$ such that
\[
	1\leq \ratio(\pm v)<\epsilon_0^6.
\]
Every such orbit thus contains a unique element such that
\[
	1\leq \ratio(v)< \epsilon_0^6\quad\text{and}\quad x-\theta_+y>0.
\]
By \eqref{eqn:discactiononVQ} and \eqref{eqn:ratio}, there is an element $\vQ=\vect{\xQ\\\yQ}\in\VQ_\RR$ such that $\Delta(\vQ)=\ratio(\vQ)=1$. After possibly acting by $g(1,-1)$, we may further assume that $\xQ-\theta_+\yQ>0$.

Let
\[
	\FQ_0:=\{g(\lambda,\alpha):\epsilon_0^{-1}<\alpha\leq1\}
\]
and
\[
	\FQ:=P\FQ_0P^{-1}.
\]
We leave the proof of the following straightforward result to the reader.
\begin{lemma}
	The set $\FQ_0$ is a fundamental domain for the action of $\langle\pm1,g(1,\epsilon_0)\rangle$ on $\GO^0_{Q_0}(\RR)$ and so $\FQ$ is a fundamental domain for the action of $\SO_Q(\ZZ)$ on $\GO^0_Q(\RR)$.
\end{lemma}

For $X>0$, let
\[
	\mc{R}_X:=\{v\in\FQ\ast\vQ:|\Delta(v)|<X\}
\]
and, for $W\subseteq[\gamma_Q]_1$ a $\mu_Q$-continuity set, let
\[
	\mc{R}_{X,W}:=\{v\in\FQ\ast\vQ:|\Delta(v)|<X,[\sh(v)]_1\in W\}.
\]
For a (Lebesgue) measurable subset $T\subseteq V_\RR$, let $\Vol(T)$ denote its Euclidean volume (i.e.\ its usual Lebesgue measure). As in \cite{Manjul-Piper}, our counts of cubic rings and fields will be estimated by counts of lattice points in $\mc{R}_{X,W}$ by relating these counts to the volume of $\mc{R}_{X,W}$. We will require the following lemma.
\begin{lemma}
	For $X\in\RR_{>0}$ and $W\subseteq[\gamma_Q]_1$ a $\mu_Q$-continuity set,
	\[
		\Vol(\mc{R}_{X,W})=\Vol(\mc{R}_{1,W})X^{1/2}.
	\]
\end{lemma}
\begin{proof}
	By \eqref{eqn:discactiononVQ}, the element $Pg(\lambda,1)P^{-1}=g(\lambda,1)$ acting on $v\in\VQ_\RR$ scales the discriminant by $\lambda^4$. This matrix also does not affect the shape of $v$, so $\mc{R}_{X,W}=g(X^{1/4},1)\ast\mc{R}_{1,W}$. Since $\mc{R}_{X,W}$ is a nice (i.e.\ semi-algebraic) subset of a $\RR^2$, scaling it by $X^{1/4}$ scales its volume by $X^{1/2}$, as claimed.
\end{proof}

We also have the following lemma that we leave to the reader.
\begin{lemma}\label{lem:FQvQ}
	As a set
	\[
		\FQ\ast\vQ=\left\{v=\vect{x\\y}\in\VQ_\RR:1\leq\ratio(v)<\epsilon_0^6,\ x-\theta_+y>0\right\}.
	\]
\end{lemma}

\subsection{Ratio-of-volumes calculation}
In this section, we prove a crucial technical lemma that relates the Euclidean volume $\Vol(\mc{R}_{1,W})$ to the hyperbolic measure $\mu_Q(W)$ for any $\mu_Q$-continuity set in $[\gamma_Q]_1$. The key idea, following \cite[\S6]{Manjul-Piper}, is to pass through the group $\GO^0_Q(\RR)$ which is linked to both.

\begin{lemma}\label{lem:RofV}
	For any $\mu_Q$-continuity set $W$ in $[\gamma_Q]_1$, we have that
	\[
		\dfrac{\Vol(\mc{R}_{1,W})}{\Vol(\mc{R}_1)}=\dfrac{\mu_Q(W)}{\mu_Q([\gamma_Q]_1)}.
	\]
\end{lemma}
\begin{proof}
	Let $\chi_{\mc{R}_{1,W}}$ denote the characteristic function of $\mc{R}_{1,W}$. Then,
	\begin{align*}
		\dfrac{\Vol(\mc{R}_{1,W})}{\Vol(\mc{R}_1)}&=\dfrac{\ds\int_{g\in\FQ}\chi_{\mc{R}_{1,W}}(g\ast\vQ)J(g)dg}{\ds\int_{g\in\FQ}\chi_{\mc{R}_{1}}(g\ast\vQ)J(g)dg},
	\end{align*}
	where $J(g)$ is an appropriate function accounting for the Jacobian determinant of the map $\FQ\rightarrow\mc{R}_\infty$, and $dg$ denotes the Haar measure on $\GO^0_Q(\RR)$. Recall that
	\[
		\GO^0_Q(\RR)=P\{g(\lambda,\alpha):\lambda\in\RR_{>0},\alpha\in\RR^\times\}P^{-1}\cong\RR_{>0}\times\RR^\times,
	\]
	so that $dg=d^\times\lambda d^\times\alpha$. The Jacobian of multiplication by a (constant) matrix is the matrix itself, the Jacobian of an inverse function is the inverse of the Jacobian of the function, and the Jacobian of a composition is the product of the Jacobians, so that the Jacobian determinants of $P$ and $P^{-1}$ cancel. To determine $J(g)$, it thus suffices to determine $\left|\dfrac{\partial(x,y)}{\partial(\lambda,\alpha)}\right|$, where
	\begin{align*}
		x&=\lambda\alpha^{-3}\xQ\\
		y&=\lambda\alpha^3\yQ.
	\end{align*}
	A straightforward calculation shows that
	\[
		\left|\dfrac{\partial(x,y)}{\partial(\lambda,\alpha)}\right|=6\xQ\yQ\frac{\lambda}{\alpha},
	\]
	which implies that $J(g)=\det(g)$. With this is hand, we have that
	\begin{align*}
		\dfrac{\Vol(\mc{R}_{1,W})}{\Vol(\mc{R}_1)}	&=\dfrac{\ds\int_{g(\lambda,\alpha)\in\FQ_0}\chi_{\mc{R}_{1,W}}(Pg(\lambda,\alpha)P^{-1}\ast\vQ)\lambda^2d^\times\lambda d^\times\alpha}{\ds\int_{g(\lambda,\alpha)\in\FQ_0}\chi_{\mc{R}_{1}}(Pg(\lambda,\alpha)P^{-1}\ast\vQ)\lambda^2d^\times\lambda d^\times\alpha}.
	\end{align*}
	By \eqref{eqn:discactiononVQ}, $\Delta(Pg(\lambda,\alpha)P^{-1}\ast\vQ)=\lambda^4\Delta(\vQ)=\lambda^4$. The characteristic functions therefore impose the condition $0<\lambda\leq1$. We also know that the $\lambda$ factor does not affect the shape, so
	\begin{align*}
		\dfrac{\Vol(\mc{R}_{1,W})}{\Vol(\mc{R}_1)}		&=\dfrac{\ds\int_0^1\lambda^2d^\times\lambda\int_{\epsilon_0^{-1}}^1\chi_{\mc{R}_{1,W}}(Pg(1,\alpha)P^{-1}\ast\vQ)d^\times\alpha}{\ds\int_0^1\lambda^2d^\times\lambda\int_{\epsilon_0^{-1}}^1\chi_{\mc{R}_{1}}(Pg(1,\alpha)P^{-1}\ast\vQ)d^\times\alpha}\\
											&=\dfrac{\ds\int_{\epsilon_0^{-1}}^1\chi_{\mc{R}_{1,W}}(Pg(1,\alpha)P^{-1}\ast\vQ)d^\times\alpha}{\ds\int_{\epsilon_0^{-1}}^1\chi_{\mc{R}_{1}}(Pg(1,\alpha)P^{-1}\ast\vQ)d^\times\alpha}.
	\end{align*}
	By Lemma~\ref{lem:FQvQ}, the value of $\chi_{\mc{R}_1}$ is constant on the region of integration here.
	On the other hand, the value of ${\chi_{\mc{R}_{1,W}}(Pg(1,\alpha)P^{-1}\ast\vQ)}$ simply depends on whether the shape of $Pg(1,\alpha)P^{-1}\ast\vQ$ is in $W$.
	The shape of $\vQ$ is some element $\zQ\in[\gamma_Q]_1$. There is therefore some $\alpha_0$ such that ${\zQ=Pg(1,\alpha_0)\ast i}$. Let $\wt{W}\subseteq\mc{T}_Q$ be the inverse of image of $W$ under the map $\pi_Q$ from \eqref{eqn:piQ}. Since
	\[
		[\sh(Pg(1,\alpha)P^{-1}\ast\vQ)]_1=\pi_Q(g(1,\alpha)g(1,\alpha_0)),
	\]
	we have that
	\begin{align*}
		[\sh(Pg(1,\alpha)P^{-1}\ast\vQ)]_1\in W	&\text{ if and only if }g(1,\alpha)g(1,\alpha_0)\in \wt{W}\\
										&\text{ if and only if }g(1,\alpha)\in g(1,\alpha_0^{-1})\wt{W}
	\end{align*}
	(where we are identifying $\mc{T}_Q$ with $\mc{T}/\langle g(1,\epsilon_0^{-1})\rangle$). We now have that
	\begin{align*}
		\dfrac{\Vol(\mc{R}_{1,W})}{\Vol(\mc{R}_1)}		&=\dfrac{\ds\int_{\epsilon_0^{-1}}^1\chi_{g(1,\alpha_0^{-1})\wt{W}}(g(1,\alpha))d^\times\alpha}{\ds\int_{\epsilon_0^{-1}}^1d^\times\alpha}.
	\end{align*}
	The measure on $\mc{T}$ is invariant under multiplication, so we may replace $\chi_{g(1,\alpha_0^{-1})\wt{W}}$ with $\chi_{\wt{W}}$. Note that $\chi_{\wt{W}}(g(1,\alpha))=\chi_W(\pi_Q(g(1,\alpha))$. Using \eqref{eqn:transfer_integral}, we obtain
	\begin{align*}
		\dfrac{\Vol(\mc{R}_{1,W})}{\Vol(\mc{R}_1)}		&=\dfrac{\frac{2}{n_Q}\ds\int_{[\gamma_Q]_1}\chi_{W}(z)d\mu_Q(z)}{\frac{2}{n_Q}\ds\int_{[\gamma_Q]_1}d\mu_Q(z)}\\
											&=\dfrac{\mu_Q(W)}{\mu_Q([\gamma_Q]_1)},
	\end{align*}
	as desired.
\end{proof}

\subsection{Equidistribution for oriented complex cubic orders}
The next step is to count the number of oriented complex cubic orders of bounded discriminant and shape in some region by relating it to the volume of $\mc{R}_{X,W}$. Together with Lemma~\ref{lem:RofV}, this will prove equidistribution for shapes of (oriented) complex cubic orders. The result for fields will then follow, in the next section, from a sieve.

For an $\SO_Q(\ZZ)$-stable subset $S$ of $\VQ_\ZZ$ and $X>0$, let
\[
	N^\Or(S; X):=\#\{[v]_1\in\SO_Q(\ZZ)\backslash S:v\text{ is irreducible, }|\Delta(v)|<X\}
\]
and for $W\subseteq[\gamma_Q]_1$ a $\mu_Q$-continuity set, let
\[
	N^\Or(S; X,W):=\#\{[v]_1\in \SO_Q(\ZZ)\backslash S:v\text{ is irreducible, }|\Delta(v)|<X,[\sh(v)]_1\in W\}.
\]
For a subset $S\subseteq\ZZ^2\subseteq V_\RR$, let $\mu_p(S)$ denote its $p$-adic density, i.e.\ the measure of its closure in $\ZZ_p$ (where $\ZZ_p$ is given its usual Haar probability measure).
In \cite[p.~74]{Manjul-Ari}, Bhargava and Shnidman prove\footnote{This proof is in some sense implicit in their discussion and is given as an analogue of their Lemma~24 and Theorem~25.} the analogue of \cite[Theorem~8]{Manjul-Piper}, namely the following.
\begin{theorem}[\cite{Manjul-Ari}]
	If $S\subseteq\VQ_\ZZ$ is any $\SO_Q(\ZZ)$-stable subset defined by finitely many congruence conditions, then
	\begin{align*}
		N^{\Or}(S;X)&=\prod_p\mu_p(S)\cdot\Vol(\mc{R}_X)+O(X^{1/4})\\
					&=\prod_p\mu_p(S)\cdot\Vol(\mc{R}_1)X^{1/2}+O(X^{1/4}).\\
	\end{align*}
\end{theorem}

The volume calculation starting on p.~73 of \cite{Manjul-Ari} shows that
\[
	\Vol(\mc{R}_X)=\dfrac{3\sqrt{3}rt\log\epsilon_0}{D}X^{1/2}
\]
and \cite[Lemma~26]{Manjul-Ari} shows that
\[
	\prod_p\mu_p(\VQ_\ZZ)=\dfrac{1}{3^{\alpha_D}rt},
\]
where
\[
	\alpha_D:=	\begin{cases}
					1&3\mid D\\
					2&3\nmid D.
				\end{cases}
\]
Define
\begin{equation}
	C_D:=\dfrac{3\sqrt{3}\log\epsilon_0}{3^{\alpha_D}D}.
\end{equation}
We therefore have the following consequence which is \cite[Theorem~29]{Manjul-Ari}: the number of oriented complex cubic orders (up to isomorphism) whose primitive trace-zero form is $\SO_Q(\ZZ)$-equivalent to $Q$ is
\[
	N^{\Or}(\VQ_\ZZ;X)=C_DX^{1/2}+O(X^{1/4}).
\]

To obtain an analogue of this result for cubic orders whose shape lies in some $W$, we will use what we call the ``$W$--$\ol{W}$ trick'' from \cite[\S3]{Manjul-Piper}. We will require the following lemma.
\begin{lemma}\label{lem:bounded}
	Let $H$ be a bounded measurable subset of $\VQ_\RR$ and let $S\subseteq\VQ_\ZZ$ be given by finitely many congruence conditions. Then, for any $z\in\RR_{>0}$, the number of irreducible lattice points in $S\cap zH$ is
	\[
		\left(\prod_{p}\mu_p(S)\right)\Vol(zH)+O(z).
	\]
\end{lemma}
\begin{proof}
As in \cite[Lemma~6 and Lemma~9]{Manjul-Piper}, what is required is to show that the number of \textit{reducible} lattices points in $S\cap zH$ is negligible. Bhargava--Shnidman prove that the number of $\SL_2(\ZZ)$-equivalence classes of binary cubic forms $F$ whose Hessian is a (positive) integer multiple of $Q$ with $|\Delta(F)|<X$ is $O(X^{1/4})$.\footnote{This is the analogue of \cite[Lemma~24]{Manjul-Ari} mentioned on page 74 of \textit{ibid.}} Since the discriminant of $(b,c)\in\VQ_\RR$ is homogeneous of degree $4$ in $b$ and $c$, for any bounded, measurable subset $H$ in $V_\RR$, the number of irreducible lattice points in $zH$ is $\Vol(zH)+O(z)$, as claimed.
\end{proof}

\begin{proposition}\label{prop:OrSWcount}
	Let $W\subseteq[\gamma_Q]_1$ be a $\mu_Q$-continuity set. Suppose that $S\subseteq\VQ_\ZZ$ is given by finitely many congruence conditions, then
	\[
		N^\Or(S;X,W)=\left(\prod_{p}\mu_p(S)\right)\cdot\Vol(\mc{R}_{1,W})X^{1/2}+O(X^{1/4}).
	\]
\end{proposition}
\begin{proof}
	Let $\epsilon>0$ and let $\mc{R}^\prime_{1,W}$ be a bounded, measurable subset of $\mc{R}_{1,W}$ such that
	\[
		\Vol(\mc{R}^\prime_{1,W})\geq\Vol(\mc{R}_{1,W})-\epsilon.
	\]
	Define $R^\prime_{X,W}:=X^{1/4}R^\prime_{1,W}$. By Lemma~\ref{lem:bounded}, the number of irreducible points in $S\cap\mc{R}^\prime_{X,W}$ is
	\[
		\left(\prod_{p}\mu_p(S)\right)\Vol(R^\prime_{1,W})X^{1/2}+O(X^{1/4}).
	\]
	Since $\mc{R}_{X,W}\supseteq\mc{R}^\prime_{X,W}$, we have that
	\[
		N^\Or(S;X,W)\geq\left(\prod_{p}\mu_p(S)\right)(\Vol(\mc{R}_{1,W})-\epsilon)X^{1/2}+O(X^{1/4}).
	\]
	This is true for all $\epsilon>0$, so that, in fact,
	\begin{equation}\label{eqn:W}
		N^\Or(S;X,W)\geq\left(\prod_{p}\mu_p(S)\right)\Vol(\mc{R}_{1,W})X^{1/2}+O(X^{1/4}).
	\end{equation}
	
	Let $\ol{W}:=[\gamma_Q]_1\setminus W$ be the complement of $W$. We similarly obtain that
	\begin{equation}\label{eqn:olW}
		N^\Or(S;X,\ol{W})\geq\left(\prod_{p}\mu_p(S)\right)\Vol(\mc{R}_{1,\ol{W}})X^{1/2}+O(X^{1/4}).
	\end{equation}
	Then,
	\begin{align}
		\left(\prod_{p}\mu_p(S)\right)\Vol(\mc{R}_{1})X^{1/2}+O(X^{1/4})&=N(S;X)\nonumber\\	&=N^\Or(S;X,W)+N^\Or(S;X,\ol{W})\nonumber\\
																&\geq\Vol(\mc{R}_{1,W})X^{1/2}+\Vol(\mc{R}_{1,\ol{W}})X^{1/2}+O(X^{1/4})\label{eqn:WolW}\\
																&=\left(\prod_{p}\mu_p(S)\right)\Vol(\mc{R}_{1})X^{1/2}+O(X^{1/4}).\nonumber
	\end{align}
	Thus, the inequality \eqref{eqn:WolW}, and hence also \eqref{eqn:W} and \eqref{eqn:olW}, are all equalities.
\end{proof}

As a corollary, we obtain the equidistribution for shapes of oriented complex cubic orders. Indeed, for $X>0$ and a $\mu_Q$-continuity set $W\subseteq[\gamma_Q]_1$, let $N^\Or_{\mathrm{rings}}(Q;X, W)$ denote the number of oriented complex cubic orders $R$ (up to isomorphism) with $[T_R^{\perp\prime}]_1=[Q]_1$, $|\Delta(R)|<X$, and $[\sh(R)]_1\in W$. Then, $N^\Or_{\mathrm{rings}}(Q;X, W)=N^\Or(\VQ_\ZZ;X,W)$. We thus obtain the following.
\begin{theorem}\label{thm:equidorientedorders}
For every $\mu_Q$-continuity set $W\subseteq[\gamma_Q]_1$,
\begin{equation}\label{eqn:equidorders}
	\lim_{X\rightarrow\infty}\frac{N^\Or_{\mathrm{rings}}(Q;X, W)}{N^\Or_{\mathrm{rings}}(Q;X)}=\frac{\mu_Q(W)}{\mu_Q([\gamma_Q]_1])}.
\end{equation}
\end{theorem}

\subsection{Equidistribution for oriented complex cubic fields}
Turning \eqref{eqn:equidorders} into a statement about \textit{maximal} cubic orders requires a sieve adapted from \cite[\S5]{DH}, as in \cite[\S5]{Manjul-Piper}. According to \cite{Manjul-Ari}, the indefinite integral binary quadratic forms arising as the primitive trace-zero forms of some cubic field are exactly those with discriminant $D$ such that $D$ or $-D/3$ is a fundamental discriminant. Those whose $D$ is not square correspond to closed geodesics, while $D=1$ or $9$ correspond to vertical non-closed geodesics. We exclude the cases $D=1$ and $D=9$ as they correspond to the pure cubic fields and have been dealt with in \cite{PureCubicShapes} (and must, in fact, be dealt with separately since the geodesics have infinite measure). We therefore assume $Q$ is such that $D$ or $-D/3$ is fundamental and $D$ is non-square.

Let $N^\Or(Q;X, W)$ denote the number of oriented complex cubic fields $K$ (up to isomorphism) with $[\TZP]_1=[Q]_1$, $|\Delta(K)|<X$, and $[\sh(K)]_1\in W$. As described in \cite[\S4.3]{Manjul-Ari}, for a prime $p$, those $v=(b,c)\in\VQ_\ZZ$ that are $p$-maximal are defined by congruences modulo $p^2$. Let $S_p\subseteq\VQ_\ZZ$ denote the set of elements that correspond to $p$-maximal rings. For $Y>2$, let
\[
	S_{<Y}:=\bigcap_{\substack{p\text{ prime}\\ p<Y}}S_p
\]
and let
\[
	S_{\mathrm{max}}:=\bigcap_{p\text{ prime}}S_p.
\]
Then, $N^\Or(Q;X, W)=N^\Or(S_{\mathrm{max}};X,W)$. Since $S_{\mathrm{max}}$ is given by infinitely many congruence conditions, Proposition~\ref{prop:OrSWcount} does not apply. However, for each $Y$, $S_{<Y}$ is given by finitely many congruences conditions, so we have that
\[
	N^\Or(S_{<Y};X, W)=\left(\prod_p\mu_p(S_{<Y})\right)\cdot\Vol(\mc{R}_{1,W})X^{1/2}+O(X^{1/4}).
\]
Let $\ol{S}_p:=\VQ_\ZZ\setminus S_p$ be the complement of $S_p$. By the discussion on page~78 of \cite{Manjul-Ari},
\[
	N^\Or(\ol{S}_p;X)=O(X^{1/2}/p^2).
\]
We may now follow the sieve of \cite[\S5]{DH}.
\begin{theorem}
	For any $\mu_Q$-continuity set $W\subseteq[\gamma_Q]_1$,
	\[
		N^\Or(Q;X,W)=\left(\prod_p\mu_p(S^{\text{max}})\right)\Vol(\mc{R}_{1,W})X^{1/2}+o(X^{1/2}).
	\]
\end{theorem}
\begin{proof}
	For any positive integer $Y$, $N(S^{\text{max}};X,W)\leq N(S_{<Y};X,W)$ so that 
	\begin{align*}
		\limsup_{X\rightarrow\infty}\dfrac{N(S^{\text{max}};X,W)}{X^{1/2}}	&\leq\lim_{Y\rightarrow\infty}\lim_{X\rightarrow\infty}\dfrac{N(S_{<Y};X,W)}{X^{1/2}}\\
																	&\leq\left(\prod_{p}\mu_p(S^{\text{max}})\right)\Vol(\mc{R}_{1,W}).
	\end{align*}
	Conversely,
	\[
		S_{<Y}\cap\mc{R}_{X,W}\subseteq(S^{\text{max}}\cap\mc{R}_{X,W})\cup\bigcup_{p\geq Y}\ol{S}_p
	\]
	so that
	\begin{align*}
		\liminf_{X\rightarrow\infty}\dfrac{N(S^{\text{max}};X,W)}{X^{1/2}}	&\geq\lim_{Y\rightarrow\infty}\lim_{X\rightarrow\infty}\dfrac{N(S_{<Y};X,W)}{X^{1/2}}-O\left(\sum_{p\geq Y}\dfrac{N^\Or(\ol{S}_p;X)}{X^{1/2}}\right)\\
															&\geq\lim_{Y\rightarrow\infty}\lim_{X\rightarrow\infty}\dfrac{N(S_{<Y};X,W)}{X^{1/2}}-O\left(\sum_{p\geq Y}\dfrac{1}{p^2}\right)\\
																	&\geq\left(\prod_{p}\mu_p(S^{\text{max}})\right)\Vol(\mc{R}_{1,W})
	\end{align*}
	since $\ds\sum_{p\geq Y}p^{-2}$ is the tail of a convergent series.
\end{proof}

This, together with Lemma~\ref{lem:RofV}, implies the equidistribution of shapes of oriented complex cubic fields on $[\gamma_Q]_1$.
\begin{corollary}\label{cor:equidorientedfields}
	For any $\mu_Q$-continuity set $W\subseteq[\gamma_Q]_1$,
	\[
		\lim_{X\rightarrow\infty}\frac{N^\Or(Q;X,W)}{N^\Or(Q;X)}=\frac{\mu_Q(W)}{\mu_Q([\gamma_Q]_1)}.
	\]
\end{corollary}

\subsection{Equidistribution for non-oriented cubic fields}
Finally, we can use the result of the previous section to prove Theorem~\ref{thm:Equidistribution}, i.e.\ equidistribution for (non-oriented) cubic fields and $\gl_2(\ZZ)$-equivalence, expanding upon the discussion of \cite[pp.~55--56]{Manjul-Ari}.

For concreteness in this discussion, we view $[\gamma_Q]$ and $[\gamma_Q]_1$ as subsets of the Gauss fundamental domains $\mc{G}$ and $\mc{G}_1$ of \eqref{eqn:G} and \eqref{eqn:G1}, respectively. For $W\subseteq[\gamma_Q]$ a $\mu_Q$-continuity set\footnote{By abuse of notation, we use $\mu_Q$ to denote the hyperbolic measures on both $[\gamma_Q]_1$ and $[\gamma_Q]$.} and $X>0$, let $N(Q;X, W)$ denote the number of (isomorphism classes of) complex cubic fields $K$ with $[\TZP]=[Q]$, $|\Delta(K)|<X$, and $[\sh(K)]\in W$. Similarly, for $N_{\mathrm{rings}}(Q;X,W)$ and orders in complex cubic fields.

First note that each isomorphism class of complex cubic orders corresponds to two isomorphism classes of oriented complex cubic orders (because of the two possible orderings of a basis). If $Q$ is ambiguous, then $[Q]_1=[Q]$. Otherwise, $[Q]$ is the union of two $\SL_2(\ZZ)$-equivalence classes $[Q]_1$ and $[wQ]_1$, where
\[
	w:=\vect{-1\\&1}.
\]
Geometrically, when $Q$ is ambiguous $[\gamma_Q]_1$ is the mirror image of itself across the imaginary axis, so that the map $[z]_1\mapsto[z]$ gives a double cover of $[\gamma_Q]$ by $[\gamma_Q]_1$. Therefore,
\begin{equation}\label{eqn:ambmuQ}
	\mu_Q([\gamma_Q])=\frac{\mu_Q([\gamma_Q]_1)}{2}.
\end{equation}
When $Q$ is not ambiguous, this is not the case; instead $[\gamma_Q]_1$ is the mirror image of $[\gamma_{wQ}]_1$. The map $[z]_1\mapsto[z]$ then gives a double cover from $[\gamma_Q]_1\cup[\gamma_{wQ}]_1$ to $[\gamma_Q]$. Therefore,
\begin{equation}\label{eqn:nonambmuQ}
	\mu_Q([\gamma_Q])=\frac{\mu_Q([\gamma_Q]_1)+\mu_{wQ}([\gamma_{wQ}]_1)}{2}=\mu_Q([\gamma_Q]_1).
\end{equation}
To deal with the translation between oriented and non-oriented orders, we must therefore deal with two cases depending on whether $Q$ is ambiguous.

Suppose first that $Q$ is not ambiguous. Let $R$ be a complex cubic order contained in a non-pure cubic field and suppose $[T_R^{\perp\prime}]=[Q]$. Then, $R$ corresponds to two isomorphism classes of oriented cubic orders $(R,\delta_1)$ and $(R,\delta_2)$. Let us denote their shapes $[\sh_1]_1$ and $[\sh_2]_1$, respectively. Then, without loss of generality, $[\sh_1]\in[\gamma_Q]_1$ and $[\sh_2]\in[\gamma_{wQ}]_1$ (and $[\sh1]=[\sh_2]=[\sh(R)]$). Suppose $W\subseteq[\gamma_Q]$ is a $\mu_Q$-continuity set. There are continuity sets $W_1\subseteq[\gamma_Q]_1$ and $W_2\subseteq[\gamma_{wQ}]_1$ such that $\wt{W}:=W_1\cup W_2\subseteq[\gamma_Q]_1\cup[\gamma_{wQ}]_1$ is the double-cover of $W$ under the map $[z]_1\mapsto[z]$ (and note that, as subsets of $\mc{G}$, $W_2=w\ast W_1$, so that $\mu_Q(W_1)=\mu_{wQ}(W_2)$). Then,
\begin{equation}\label{eqn:nonambmuQW}
	\mu_Q(W)=\frac{\mu_Q(W_1)+\mu_{wQ}(W_2)}{2}=\mu_Q(W_1).
\end{equation}
Note that
\[
	[\sh(R)]\in W\quad \text{if and only if}\quad[\sh_1]_1\in W_1\text{ and }[\sh_2]_1\in W_2
\]
so that
\[
	N(Q;X,W)=N^\Or(Q;X,W_1)=N^\Or(wQ;X,W_2)
\]
(and similarly for $N_{\mathrm{rings}}(Q;X,W)$).
Then,
\begin{equation}\label{eqn:unorientedtoorientedratio}
	\frac{N(Q;X,W)}{N(Q;X)}=\frac{N^\Or(Q;X,W_1)}{N^\Or(Q;X)}.
\end{equation}

Now, suppose that $Q$ is ambiguous. As in the previous paragraph, we get $(R,\delta_i)$ and $[\sh_i]_1$, for $i=1,2$. Without loss of generality, suppose $[\sh_1]_1\in\mc{G}$. Let $\wt{W}$ be the inverse image of $W$ under the double-cover $[\gamma_Q]_1\rightarrow[\gamma_Q]$. Then, $\wt{W}=W\cup w\ast W$, so that
\begin{equation}\label{eqn:ambmuQW}
	\mu_Q(W)=\frac{\mu_Q(\wt{W})}{2}.
\end{equation}
Now,
\[
	[\sh(R)]\in W\quad \text{if and only if}\quad[\sh_1]_1\in W\text{ and }[\sh_2]_1\in w\ast W.
\]
So, it certainly looks like we are double-counting everything when considering oriented rings and $\SL_2(\ZZ)$-equivalence. It is however possible that sometimes $[\sh_1]_1=[\sh_2]_1$. Luckily, this only happens when $[\sh(R)]$ lies on the boundary of $\mc{G}$ and Theorem~\ref{thm:avoidboundaries} shows that this does not occur for shapes of orders in non-pure complex cubic fields! We thus have that
\begin{equation}\label{eqn:unorientedtooriented}
	N(Q;X,W)=\frac{1}{2}N^\Or(Q;X,\wt{W})
\end{equation}
(and similarly for $N_{\mathrm{rings}}(Q;X,W)$).

Therefore, when $Q$ is not ambiguous, combining \eqref{eqn:nonambmuQ}, \eqref{eqn:nonambmuQW}, and \eqref{eqn:unorientedtoorientedratio} with Corollary~\ref{cor:equidorientedfields}, we have that
\[
	\lim_{X\rightarrow\infty}\frac{N(Q;X,W)}{N(Q;X)}=\lim_{X\rightarrow\infty}\frac{N^\Or(Q;X,W_1)}{N^\Or(Q;X)}=\frac{\mu_Q(W_1)}{\mu_Q([\gamma_Q]_1)}=\frac{\mu_Q(W)}{\mu_Q([\gamma_Q])}.
\]
And when $Q$ is ambiguous, we can similarly combine \eqref{eqn:ambmuQ}, \eqref{eqn:ambmuQW}, and \eqref{eqn:unorientedtooriented} with Corollary~\ref{cor:equidorientedfields} to obtain that
\[
	\lim_{X\rightarrow\infty}\frac{N(Q;X,W)}{N(Q;X)}=\lim_{X\rightarrow\infty}\frac{\frac{1}{2}N^\Or(Q;X,\wt{W})}{\frac{1}{2}N^\Or(Q;X)}=\frac{\frac{1}{2}\mu_Q(\wt{W})}{\frac{1}{2}\mu_Q([\gamma_Q]_1)}=\frac{\mu_Q(W)}{\mu_Q([\gamma_Q])}.
\]
We have therefore proved Theorem~\ref{thm:Equidistribution}, as well as its analogue for orders which he state here.
\begin{theorem}\label{thm:equidorders}
	For every $\mu_Q$-continuity set $W\subseteq[\gamma_Q]$,
\[
	\lim_{X\rightarrow\infty}\frac{N_{\mathrm{rings}}(Q;X, W)}{N_{\mathrm{rings}}(Q;X)}=\frac{\mu_Q(W)}{\mu_Q([\gamma_Q])}.
\]
\end{theorem}

\section{Generalization to higher degree}\label{sec:higherdegree}
In this section, we prove Theorem~\ref{thm:higherdegree}, or rather a slight generalization: the shape of any order $\mc{O}$ in a finite \'{e}tale $\QQ$-algebra lies on the majorant space of the trace-zero form of $\mc{O}$. This gives a vast generalization of Theorem~\ref{thm:3TKonMK}.

Let $\A$ be an $n$-dimensional \'{e}tale $\QQ$-algebra. More concretely, $\A\cong K_1\times\cdots\times K_r$, where the $K_i$ are algebraic number fields. Let $d_i:=\deg K_i$, so that $n=\ds\sum_{i=1}^rd_i$. Each $K_i$ has $d_i$ embeddings into $\CC$ which we will denote $\sigma_{i,1},\cdots,\sigma_{i,d_i}$. Then, as in the case of number fields, we have an embedding
\[
	\funcdef{j}{\A}{\CC}{(a_1,\dots,a_r)}{(\sigma_{i,k}(a_i))_{i,k}}
\]
that embeds $\A$ into an $n$-dimensional Euclidean space $\A_\RR$. Let $\mc{O}$ be an order in $\A$. We can then once again speak of the shape $\MO$ and the trace-zero form $\TO$ of $\mc{O}$. Let $r_1$ be the number of real embeddings of $\A$ into $\CC$ and $r_2$ be the number of pairs of complex embeddings. A result of Olga Taussky-Todd \cite{Taussky} shows that the trace form of $\mc{O}$ has signature $(r_1+r_2,r_2)$. The trace-zero form $\TO$ then has signature $(r_1+r_2-1,r_2)$. Let $O(\TO)$ be the orthogonal group of $\TO$. Let $T$ be the Gram matrix of $\TO$ (with respect to some integral basis of $\mc{O}^\perp$). Siegel \cite[\S3.2]{Siegel} defines the \textit{majorant space}\footnote{Siegel refers to it as the $\mf{H}$-space.} $\mf{H}_\mc{O}$ of $T$ as the set of positive definite matrices $M$ such that
\[
	MT^{-1}M=T.
\]
This is a model for the symmetric space of $O(\TO)$. A direct generalization of Theorem~\ref{thm:3TKonMK} would say that the shape of $\mc{O}$ ``lies on'' $\mf{H}_\mc{O}$. One can verify ``by hand'' that this equality holds for the Gram matrices \eqref{eqn:MGram} and \eqref{eqn:TGram} in the cubic case, thus providing another proof of Theorem~\ref{thm:3TKonMK}. By very slightly modifying the work of \cite{Taussky}, we show this holds in full generality.

\begin{theorem}
	Let $T$ and $M$ be the Gram matrices of $\TO$ and $\MO$, respectively, with respect to some integral basis of $\mc{O}^\perp$. Then,
	\[
		MT^{-1}M=T.
	\]
\end{theorem}
\begin{proof}
	In fact, we will prove this equality holds even before taking orthogonal projections. Let $\alpha_0,\alpha_1,\dots,\alpha_{n-1}$ be an integral basis of $\mc{O}$ and let $\wt{T}$ and $\wt{M}$ be the Gram matrices of the trace form and the Minkowski inner product, respectively, with respect to this basis. Let $v_i=j(\alpha_i)\in\CC^n$. We can think of the $v_i$ as column vectors and let $A$ be the $n\times n$-matrix
	\[
		A=\vect{v_0&v_1&\cdots&v_{n-1}}.
	\]
	Then,
	\begin{equation}\label{eqn:MandT}
		\wt{M}=\ol{A}^TA\quad\text{and}\quad\wt{T}=A^TA.
	\end{equation}
	
	First, note that the equality we wish to prove is independent of the choice of basis. Indeed, if $P\in\gl_n(\RR)$, $\wt{M}_1=P^T\wt{M}P$, $\wt{T}_1=P^T\wt{T}P$, and $\wt{M}_1\wt{T}_1^{-1}\wt{M}_1=\wt{T}_1$, then
	\[
		P^T\wt{T}P=\wt{T}_1=\wt{M}_1\wt{T}_1^{-1}\wt{M}_1=P^T\wt{M}\wt{T}^{-1}\wt{M}P
	\]
	so that $\wt{T}=\wt{M}\wt{T}^{-1}\wt{M}$ since $P$ is invertible.
	
	Second, note that what we just explained shows that the identity being true for $\wt{M}$ and $\wt{T}$ implies its truth for $M$ and $T$. Indeed, let us choose $\alpha_0=1$ (which is always possible), then the matrix bringing $v_0,v_1,\dots,v_{n-1}$ to $v_0,v_1^\perp,\dots,v_{n-1}^\perp$ is some $P\in\gl_n(\RR)$, and $P^T\wt{M}P=1\oplus M$ and $P^T\wt{T}P=1\oplus T$.

Now, we proceed as in \cite{Taussky}, including the details here for the convenience of the reader. The matrix $A$ has $r_1$ real rows and $r_2$ pairs of complex conjugate rows. There is therefore an explicit $R\in\gl_n(\RR)$ such that $RA$ has the same real rows and each pair of complex conjugate rows with entries $a_k+ib_k,a_k-ib_k$, respectively, becomes a pair of rows with entries $a_k,ib_k$, respectively. After possibly reordering the rows (which amounts to replacing $R$ with another matrix in $\gl_n(\RR)$), we may assume that the first $r_1+r_2$ rows of $RA$ are real and that the last $r_2$ are purely imaginary. Let $D$ be the diagonal matrix whose first $r_1+r_2$ entries are $1$ and whose remaining $r_2$ entries are $i$. Then, $DRA$ is some invertible real matrix whose inverse we will denote by $P$. Write
	\[
		R^{-T}R^{-1}=\vect{B_{11}&B_{12}\\B_{21}&B_{22}}
	\]
	where $B_{11}$ is an $(r_1+r_2)\times(r_1+r_2)$-matrix and $B_{22}$ is an $r_2\times r_2$-matrix.
	Then,
	\[
		D^{-T}R^{-T}R^{-1}D^{-1}=\vect{B_{11}&-iB_{12}\\-iB_{21}&-B_{22}}
	\]
	so that $B_{12}$ and $B_{21}$ are $0$, since $D^{-T}R^{-T}R^{-1}D^{-1}=P^TA^TAP$, which is a real matrix. Since $R^{-T}R^{-1}$ is positive definite, so are the $B_{kk}$. Let $\wt{T}_1=P^T\wt{T}P$ and $\wt{M}_1=P^T\wt{M}P$. Then, by \eqref{eqn:MandT} and what we have just showed,
	\begin{align*}
		\wt{M}_1\wt{T}_1^{-1}\wt{M}_1	&=\vect{B_{11}&\\&B_{22}}\cdot \vect{B_{11}^{-1}&\\&-B_{22}^{-1}}\cdot\vect{B_{11}&\\&B_{22}}\\
								&=\vect{B_{11}&\\&-B_{22}}\\
								&=\wt{T}_1,
	\end{align*}
	as desired.
\end{proof}

This leads to the following natural question: fix a primitive quadratic form $Q$ over $\ZZ$ in $n-1$ variables (that arises as $\TZP$ for some number field $K$) of signature $(r_1+r_2-1,r_2)$ (where $n=r_1+2r_2$), are the shapes of the degree $n$ number fields $K$ with signature $(r_1,r_2)$ whose $\TZP$ is (equivalent to) $Q$ equidistributed on $\mf{H}_Q$?

\subsection*{Acknowledgments}
The author would like to thank Manjul Bhargava, Asaf Hadari, Piper~H, and Akshay Venkatesh for some helpful conversations.

\bibliographystyle{amsalpha}

\bibliography{cubic}

\end{document}